\documentclass[leqno, 12pt]{amsart}
\usepackage{amsmath,amsfonts,amsthm,amssymb,indentfirst}
\usepackage{xy} \xyoption{all}

\newtheorem{theorem}{Theorem}[section]
\newtheorem{lemma}[theorem]{Lemma}
\newtheorem{corollary}[theorem]{Corollary}
\newtheorem{proposition}[theorem]{Proposition}

\theoremstyle{definition}
\newtheorem{definition}[theorem]{Definition}
\newtheorem{example}[theorem]{Example}

\newtheorem{question}[theorem]{Open Problem}

\newcommand{\Cset}{\mathbb{C}}
\newcommand{\Rset}{\mathbb{R}}
\newcommand{\so}{\mathbf{s}}
\newcommand{\ra}{\mathbf{r}}

\begin{document}

\title{Canonical traces and directly finite Leavitt path
algebras}

\author{Lia Va\v s}
\address{Lia Va\v s\\ Department of Mathematics, Physics and Statistics\\
University of the Sciences\\
Phila\-delphia, PA 19104, USA}
\email{l.vas@usciences.edu}

\subjclass[2000]{
16S99, 
16W99, 
16W10, 
16P99}  

\keywords{Trace, Leavitt path algebra, directly finite, involution, graph trace,
gauge invariant, positive, faithful, Cohn-Leavitt algebra}

\begin{abstract} 
Motivated by the study of traces on graph $C^*$-algebras, we consider traces
(additive, central maps) on Leavitt path algebras, the algebraic counterparts of
graph $C^*$-algebras. In particular, we consider traces which vanish on nonzero
graded components of a Leavitt path algebra and refer to them as {\em
canonical} since they are uniquely determined by their values on
the vertices. 

A desirable property of a $\Cset$-valued trace on a $C^*$-algebra is that the
trace of an element of the positive cone is nonnegative. We adapt this property
to traces on a Leavitt path algebra $L_K(E)$ with values in any involutive ring.
We refer to traces with this property as positive. If a positive trace is
injective on positive elements, we say that it is faithful. We characterize when
a canonical, $K$-linear trace is positive and when it is faithful in terms of
its values on the vertices. As a consequence, we obtain a bijective
correspondence between the set of faithful, gauge invariant, $\Cset$-valued
(algebra) traces on $L_{\Cset}(E)$ of a countable graph $E$ and the set of
faithful, semifinite, lower semicontinuous, gauge invariant (operator theory)
traces on the corresponding graph $C^*$-algebra $C^*(E)$. 

With the direct finite condition (i.e $xy=1$ implies $yx=1$) for unital rings
adapted to rings with local units, we characterize directly finite Leavitt path
algebras as exactly those having the underlying graphs in which no cycle has an
exit. Our proof involves consideration of ``local'' Cohn-Leavitt subalgebras of
finite subgraphs. Lastly,
we show that, while related, the class of locally noetherian, the class of
directly finite, and the class of Leavitt path algebras which admit a faithful
trace are different in general. 
\end{abstract}

\maketitle

\section{Introduction}

Throughout their existence, many operator theory concepts have been
subject to ``algebraization'' -- the study of algebraic counterparts of
operator theory concepts using algebraic methods alone. Regular rings, Baer
rings, and their numerous generalizations have all been created by
algebraization of some operator theory concepts. Recently, Leavitt path
algebras have joined this list as algebraic counterparts of graph
$C^*$-algebras and many properties of graph $C^*$-algebras
have been formulated for Leavitt path algebras and proven using solely algebraic
methods. Our interest in traces is greatly inspired by their relevance in
the study of noncommutative geometry of graph $C^*$-algebras from
\cite{Pask-Rennie}. 

The class of all traces, i.e. additive and central maps, on a Leavitt
path algebra is a rather large class. After some preliminaries, in section
\ref{section_graph_traces}, we restrict our attention to those traces that  
vanish on nonzero graded components of a Leavitt path algebra and refer to
them as {\em canonical traces} (Definition
\ref{definition_canonical}). Such
traces are canonical in the sense that they are completely determined by
the values on the vertices. In particular, we show that every {\em graph trace}
(a map on the vertices of the underlying graph which agrees with the (CK2) axiom
of Leavitt path algebras) uniquely extends to a canonical trace (Proposition
\ref{canonical}). A canonical trace is gauge invariant and, if the
characteristic of the underlying field is zero, the converse is true as well
(Proposition \ref{gauge_invariant}). 

In operator theory, a desirable property of a $\Cset$-valued trace on
a $C^*$-algebra is that the trace of an element of the positive cone is 
nonnegative. The algebraic version of this property for a trace $t:
R\to T$, where $R$ and $T$ are involutive rings, is that the trace of a positive
element of $R$ (a finite sum of elements of the form $xx^*$) is a positive
element of $T$. If a trace has this property we say it is {\em positive}. A
positive trace is {\em faithful} if the trace of a nonzero, positive element is
nonzero and positive.

Given a graph $E$ and a field
$K,$ consider the Leavitt path algebra $L_K(E)$ of $E$ over $K.$
\cite[Proposition 29]{Zak_paper} lists some necessary conditions, given in terms
of trace values on the vertices of $E$, for a trace on $L_K(E)$ to be positive
and faithful. In \cite{Zak_paper}, it is shown that these conditions are not
sufficient. In section \ref{section_positive_faithful}, we prove that conditions
(1)--(3) of \cite[Proposition 29]{Zak_paper} are sufficient for a canonical,
$K$-linear trace on $L_K(E)$ to be positive (Theorem \ref{positive}). If a
canonical, $K$-linear trace on $L_K(E)$ has values in a positive definite
algebra, conditions (1)--(4) are sufficient for this trace to be faithful
(Theorem \ref{faithful}). Theorems \ref{positive} and \ref{faithful} imply that
any positive graph trace on $E$ uniquely
extends to a positive, canonical, $K$-linear trace on $L_K(E)$ and, any
faithful graph trace on $E$ with values in a positive definite algebra uniquely
extends to a faithful, canonical, $K$-linear trace on $L_K(E)$ (Theorem
\ref{correpondence}). 

Let $\Cset$ denote the field of complex numbers with the
complex-conjugate involution. \cite[Proposition 3.9]{Pask-Rennie} shows that
there is a bijective correspondence between the set of faithful, $\Cset$-valued
graph traces on a countable, row-finite graph $E$ and the set of faithful,
semifinite, lower 
semicontinuous, gauge invariant, $\Cset$-valued traces (in the operator theory
sense) on the graph $C^*$-algebra $C^*(E)$. We show that the two sets above
are also in a bijective correspondence with the set of faithful, gauge
invariant, $\Cset$-linear, $\Cset$-valued traces on the Leavitt path algebra
$L_{\Cset}(E)$ and that it is not necessary to require that $E$ is row-finite
(Corollary \ref{c-star}). 

The main goal of the second part of the paper (section \ref{section_finite}) is
to characterize directly finite Leavitt path algebras by properties of the
underlying graph. Recall that a unital ring is directly finite if $xy=1$ implies
that $yx=1$ for all $x$ and $y$. We say that a ring with local units is directly
finite if for every $x, y$ and an idempotent $u$ such that $xu=ux=x$ and
$yu=uy=y,$ we have that $xy=u$ implies $yx=u.$ Inspired by results of Ara and
Goodearl in \cite{Ara_Goodearl}, we note that consideration of finitely many
elements of a Leavitt path algebra can be reduced to their consideration as
elements of a Cohn-Leavitt subalgebra of a finite subgraph. We show
that the Cohn-Leavitt path algebra of a finite graph $E$ is directly finite if
and only if no cycle of $E$ has an exit and the (CK2) axiom holds for all
vertices of the cycles. Using this result, we show that a Leavitt path algebra
of a graph $E$ is directly finite if and only if no cycle of $E$ has an exit
(Theorem \ref{no-exit}). 

Cohn-Leavitt algebras encompass both Cohn path algebras and Leavitt path
algebras and can be viewed as algebraic counterparts of relative graph
$C^*$-algebras. We adapt our previous results, Theorems \ref{positive},
\ref{faithful}, and \ref{correpondence} to Cohn-Leavitt algebras (Propositions
\ref{CL_positive_faithful} and \ref{CL_correspondence}), and use our
characterization of directly finite Leavitt path algebras to show that a Cohn
path algebra $C_K(E)$ is directly finite if and only if $E$ is acyclic
(Corollary \ref{CL_no-exit}). We also note that the properties that $L_K(E)$ is
locally noetherian and that $L_K(E)$ admits a faithful trace are independent,
both imply that $L_K(E)$ is directly finite and that both implications
are strict (Examples \ref{direct_finite_not_noetherian_no_trace},
\ref{noetherian_no_trace}, and \ref{has_trace_not_noetherian}). We conclude the
paper by considering an open problem (\ref{question1}).

\section{Positive, faithful, and canonical traces on Leavitt path algebras}
\label{section_graph_traces}

Throughout the paper, all rings are assumed to be associative, but not
necessarily unital. The notation $\delta_{a,b}$ is used to denote 1 if $a=b$
and 0 if $a\neq b$ for any set $A$ and
$a,b$ in $A.$
We start by recalling a few general definitions and establishing some
preliminary results. 

Let $R$ and $T$ be rings. A map $t: R \to T$ is   {\em central} 
if  $t(xy)=t(yx)$ for all $x,y \in R$ and it is a $T$-\emph{valued trace on}
$R$ if $t$ is an additive, central map. If $R$ and $T$ are $C$-algebras, for
some
commutative ring $C$, then the trace $t$ is
$C$-\emph{linear} if $t(cx)=ct(x)$ for all $x \in R$ and $c \in C$.

The standard trace on a matrix ring over a commutative
ring $C$ is an example of a  $C$-linear, $C$-valued trace. If $G$ is a group,
the Kaplansky trace and the augmentation map on the
group ring $CG$ are also examples of $C$-linear, $C$-valued traces. 

Recall that an {\em involution} on a ring $R$ is an additive map $* : R\to R$
such that $(xy)^*=y^*x^*$ and $(x^*)^*  = x$ for all $x,y\in R.$ In this case
$R$ is called an {\em involutive ring} or a {\em $*$-ring}. If $R$ is also a
$C$-algebra for some commutative, involutive ring $C,$ then $R$ is a {\em
$*$-algebra} if $(ax)^*=a^*x^*$ for $a\in C$ and $x\in R.$

An element of a $*$-ring
$R$ is \emph{positive} if it is a finite 
sum of elements of the form $xx^*$ for $x \in R.$ The notation $x>0$ usually
denotes positive elements. We abuse this notation slightly and denote the
fact that $x$ is positive element by $x\geq 0.$ If $x$ is positive and
nonzero, we write $x>0.$ One may argue that we should refer to positive
elements as nonnegative instead. Although this may be a valid point, we
continue to use the terminology which is well established in operator
theory and keep referring to such elements as positive. 

An involution $*$ on $R$ is \emph{positive definite} if, for all $x_1, \dots,
x_n
\in R$, $\sum_{i=1}^n x_ix_i^* = 0$ implies $x_i=0$ for each $i=1,\ldots,
n$ and it is {\em proper} if this condition holds for $n=1.$ 
A $\ast$-ring with a positive definite (proper) involution is referred to
as {\em positive definite (proper)}. By \cite[Exercise 9A, sec. 13]{Berberian},
a $*$-ring $R$ is positive definite if and only $R$ is proper and the conditions
\[x\geq 0,\; y\geq 0,\mbox{ and }x+y=0\mbox{ imply }x=y=0.\]

The relation $\geq$ extends to all elements of a $*$-ring $R$ by 
\[x\geq y\mbox{ if and only if }x-y\geq 0.\]
This relation is always reflexive and transitive (see \cite[Section
50]{Berberian}). The antisymmetry holds if $R$ is positive
definite. 

Let $R$ and $T$ be $*$-rings and $t: R\to T$ be an additive map. 
\begin{enumerate}
\item[] The map $t$ is \emph{positive} if $t(x)\geq 0$ for all $x \in R$ with $x\geq 0$ 
(equivalently $t(xx^*)\geq 0$ for all $x\in R$). 

\item[] The map $t$ is \emph{faithful}
if  $t(x) > 0$ for all $x \in R$ with $x > 0$ (equivalently $t$ is positive
and $x\geq 0$ and $t(x)=0$ imply $x=0$).
\end{enumerate}

The following lemma further characterizes faithful, additive maps with values
in a positive definite $*$-ring. 

\begin{lemma} Let $R$ and $T$ be $*$-rings, $T$ be positive definite, and 
$t: R\to T$ be any positive, additive map.
The following are equivalent. 
\begin{enumerate}
\item  The map $t$ is faithful. 
\item $x\geq 0, y\geq 0$ and $t(x+y) = 0$ imply $x = y = 0$ for all 
$x, y\in  R.$
\end{enumerate}
Conditions {\em (3)} and {\em(4)} below imply {\em(1)} and {\em(2)}. 
If $R$ is proper, then {\em(1)} and {\em(2)} are equivalent to {\em (3)}
and {\em(4)}. 
\begin{enumerate}
\item[(3)] $t(xx^*+yy^*)=0$ implies $x=y=0$ for all $x,y\in R.$
\item[(4)] $t(xx^*)=0$ implies $x=0$ for all $x\in R.$
\end{enumerate}
\label{positively_faithful}
\end{lemma}
\begin{proof}
To prove that (1) implies (2), let $x\geq 0$, $y\geq 0$ and $t(x+y)=0$
for $x,y \in R.$ Since $t$ is positive $t(x)\geq 0$ and $t(y)\geq 0.$
Since $T$ is positive definite $t(x)+t(y)=t(x+y)=0$ implies that $t(x)=t(y)=0.$ 
Then $x=y=0$ by faithfulness of $t$. 

Condition (2) with $y=0$ implies condition (1) and condition (3) with $y=0$
implies condition (4). 

Condition (4) implies (1). Indeed, if $x\geq 0$ and $t(x)=0$,
then $x=\sum_{i=1}^n a_ia_i^*$ for some $a_i\in R$ and $0=t(\sum_{i=1}^n
a_ia_i^*)=\sum_{i=1}^n t(a_ia_i^*).$ Since $T$ is positive definite and $t$ is
positive, this
implies that $t(a_ia_i^*)=0$ for all $i=1,\ldots,n.$ Then $a_i=0$ for all
$i=1,\ldots,n$ by condition (4). Thus $x=0.$ 

Assuming now that $R$ is proper and that condition (2) holds, let us show
(3). 
If $t(xx^*+yy^*)=0$, then condition (2) implies that $xx^*=0$ and $yy^*=0.$
Then $x=y=0$ by properness of $R$.
\end{proof}
 
Since any $C^*$-algebra is proper (\cite[page 11]{Berberian}) and $\Cset$ is
positive definite when equipped with the complex-conjugate involution, a
$\Cset$-valued, additive, and positive map on a $C^*$-algebra can be defined to
be faithful using any of the conditions (1)--(4).
In fact, in operator theory texts, either condition (3) or condition
(4) are frequently used when defining a faithful trace. 

We review the definition of a Leavitt path algebra now. 
Let $E=(E^0, E^1, \so_E, \ra_E)$ be a directed graph where $E^0$ is
the set of vertices, $E^1$ the set of edges, and $\so_E,\ra_E:E^1
\to E^0$ are the source and the range maps. Since we consider just directed
graphs, we refer to them simply as
graphs. Also, if it is clear from the context, we write $\so_E$ and $\ra_E$ shorter
as $\so$ and $\ra.$ A {\em path} $p$ in $E$ is a finite sequence of edges
$p=e_1\ldots e_n$ such that $\ra(e_i)=\so(e_{i+1})$ for $i=1,\dots,n-1$. Such
path $p$ has length $n$ and we write $|p|=n.$  The maps $\so$ and $\ra$ extend
to paths by $\so(p)=\so(e_1)$ and $\ra(p)=\ra(e_n)$. We consider vertices as
paths of length zero. A path $p = e_1\ldots e_n$ is said to be \emph{closed} if
$\so(p)=\ra(p)$. A closed path is said to be a \emph{cycle} if $\so(e_i)\neq
\so(e_j)$ for every $i\neq j$. A graph $E$ is said to be {\em no-exit} if
$\so^{-1}(v)$ has just one element for every vertex $v$ of every cycle. 

A vertex $v$ is said to be {\em regular} if the set $\so^{-1}(v)$ is nonempty
and finite, $v$ is called a {\em sink} if $\so^{-1}(v)$ is empty, and $v$ is
called an {\em infinite emitter} if $\so^{-1}(v)$ is infinite. A graph $E$ is
\emph{row-finite} if sinks are the only vertices that are not regular,
\emph{finite} if it is row-finite and $E^0$
is finite (in which case $E^1$ is necessarily finite as well), and {\em
countable} if both $E^0$ and $E^1$ are countable. 

For a graph $E,$ consider the extended graph of $E$ to be the graph with the
same vertices and with edges $E^1\cup \{e^*\ |\ e\in E^1\}$ where
the range and source relations are the same as in $E$ for $e\in
E^1$ and $\so(e^*)=\ra(e)$  and $\ra(e^*)=\so(e)$ for the added edges.
Extend the map $^*$ to all the paths by defining $v^*=v$ for
all vertices $v$ and $(e_1\ldots e_n)^*=e_n^*\ldots e_1^*$ for all paths
$p=e_1\ldots e_n.$ If $p$ is a path, we refer to elements of the form
$p^*$ as ghost paths. Extend also the maps $\so$ and $\ra$ to ghost paths by
$\so(p^*)=\ra(p)$ and $\ra(p^*)=\so(p)$.  

In the rest of the paper, $E$ denotes a graph and $K$ a field. 
The \emph{Cohn path algebra $C_K(E)$ of $E$ over
$K$} is the free $K$-algebra generated by $E^0\cup E^1\cup
\{e^*\ |\ e\in
E^1\}$ 
subject to the following relations for all vertices $v,w$ and edges $e,f$.
\begin{itemize}
\item[(V)]  $vw = \delta_{v,w}v$,

\item[(E1)]  $\so(e)e=e\ra(e)=e$,

\item[(E2)] $\ra(e)e^*=e^*\so(e)=e^*$,

\item[(CK1)] $e^*f=\delta _{e,f}\ra(e)$.
\end{itemize}
The four axioms above imply that every
element of $C_K(E)$ can be represented as a sum of the form $\sum_{i=1}^n
a_ip_iq_i^*$ for some $n$, paths $p_i$ and $q_i$, and elements $a_i\in K,$ for 
$i=1,\ldots,n.$ We use $G_E$ to denote the set of all elements of the form
$pq^*$ where $p$ and $q$ are paths with $\ra(p)=\ra(q).$ 

If the underlying field $K$ has an involution $*$ (and there is always at
least one such involution, the identity), the involution $*$ from
$K$ extends to an involution of $C_K(E)$ by 
$(\sum_{i=1}^n a_ip_iq_i^*)^* =\sum_{i=1}^n a_i^*q_ip_i^*$ making $C_K(E)$ a
$*$-algebra.

The \emph{Leavitt path algebra $L_K(E)$ of $E$ over
$K$} is the free $K$-algebra generated by $E^0\cup E^1\cup \{e^*\ |\ e\in
E^1\}$ 
subject to relations (V), (E1), (E2), (CK1) and 
\begin{itemize}
\item[(CK2)] $v=\sum_{e\in \so^{-1}(v)} ee^*$ for every regular vertex $v$.
\end{itemize}

The Leavitt path algebra $L_K(E)$ can also be defined as the quotient
$C_K(E)/N$ where  $N$ is the ideal of the Cohn algebra $C_K(E)$ generated by all
elements of the form $v~-~\sum_{e\in \so^{-1}(v)} ee^*$ where $v$ is a regular
vertex. The algebra $L_K(E)$ is an involutive algebra with the involution
inherited from $C_K(E).$
In some early works on Leavitt path algebras, the field $K$ was assumed to have 
the identity involution in all cases except when $K=\Cset$ in
which case the involution was assumed to be the
complex-conjugate involution. We stress the advantage of considering the base
field $K$ as an involutive field with {\em any} involution and defining the
involution on $L_K(E)$ by $$(\sum_{i=1}^n a_ip_iq_i^*)^* =\sum_{i=1}^n
a_i^*q_ip_i^*$$ for paths $p_i,q_i$ and $a_i\in K,$ $i=1,\ldots, n.$ This
approach unifies the treatment of different involutive fields and
integrates
consideration of both the identity and the complex-conjugate involution on
$\Cset.$ 

The underlying field $K$ does not play a role when characterizing many
algebraic properties of Leavitt path algebras as shown in numerous papers. In
fact, it has been hypothesized that two Leavitt
path algebras isomorphic over one field are isomorphic over
any other field. While this issue is still not settled, we
point out that the presence of an involution definitely brings the underlying
field into focus and makes properties of the field $K$
relevant for the ``involution sensitive'' ring-theoretic properties of
$L_K(E)$.
This is apparent in \cite[Theorem 3.3]{GonzaloRangaLia}
for example.  We point out
that Theorems \ref{faithful} and \ref{correpondence}
have the same sensitivity to involution on $K$. Namely, the assumptions of
Theorems \ref{faithful} and \ref{correpondence} imply
that the field $K$ is positive definite as the next proposition shows. In this
case, \cite[Proposition 2.4]{GonzaloRangaLia} shows that $L_K(E)$ is
positive definite for some (equivalently any) graph $E$.

\begin{proposition}
Let $R$ be an involutive $K$-algebra and $t:L_K(E)\to R$ a $K$-linear
map on a Leavitt
path algebra $L_K(E)$. If $R$ is positive definite and $t$ is faithful, then $K$
is positive definite thus $L_K(E)$ is positive definite as well.
\label{positively_faithful_on_LPA} 
\end{proposition}
\begin{proof}
If  $R$ is positive definite and $t$ is faithful, then $t$ satisfies
condition (2) of Lemma \ref{positively_faithful} by that Lemma.
To show that $K$ is positive definite, note first that $K$ is
proper since $K$ is a field:
if $a\neq 0$ for $a\in K$ then
$a^*\neq 0$ and so $aa^*\neq 0$ as well. Now let us assume
that 
$\sum_{i=1}^n a_ia_i^*=0$ for some $a_1,\dots, a_n \in K$.
Then for any vertex $v$ we
have that $0=\sum_{i=1}^n a_ia_i^*v=\sum_{i=1}^n (a_iv)(a_iv)^*$ and so 
$t(\sum_{i=1}^n (a_iv)(a_iv)^*)=0.$ Since condition (2) of Lemma
\ref{positively_faithful} holds, this implies that
$0=(a_iv)(a_iv)^*=a_ia_i^*v$ for each $i$. Assuming
that $a_ia_i^*\neq 0$ we would have $v=0$ which is a contradiction. Thus 
$a_ia_i^*=0$ and so $a_i=0$ since $K$ is proper. Thus $K$ is positive definite.
In this case, $L_K(E)$ is positive definite as well by \cite[Proposition
2.4]{GonzaloRangaLia}.
\end{proof}

If a trace $t:L_K(E)\to R$ is positive for some $*$-ring $R,$ 
then the trace values of vertices are positive elements of $R$. Let
us denote this condition by (P0).
\begin{enumerate}
\item[(P0)] $t(v)\geq 0$ for all vertices $v.$
\end{enumerate}
\cite[Proposition 29]{Zak_paper}
lists more necessary conditions for a trace $t$ to be positive:
\begin{enumerate}
\item[(P1)] $t(v) \geq t(w)$ for all vertices $v$ and $w$, such that there
is a path $p$ with  $\so(p) =v$ and $\ra(p) =w$.
\item[(P2)] $t(v) \geq \sum_{i=1}^n t(\ra(e_i))$ for all vertices $v$
and distinct edges $e_1, \dots, e_n$ with $v$ as the source.
\end{enumerate}

We show that (P2) implies (P1) and we combine conditions
(P0) and (P2) into a single condition.  
\begin{lemma} If $R$ is a $*$-ring and $t$ is an $R$-valued trace on
$L_K(E)$, conditions (P0)
and (P2) are equivalent to condition (P) below and imply (P1). 
\begin{enumerate}
\item[(P)] $\;\;\;t(v-\sum_{e\in I} \ra(e))\geq 0\;\;\;$ for all vertices $v$ and finite subsets $I$ of $\so^{-1}(v)$. 
\end{enumerate}
\label{condition_P} 
\end{lemma}
\begin{proof}
Condition (P) is condition (P0) in case when the set $I$ is
empty. If $I$ is nonempty, conditions (P2) and (P) are equivalent since 
$t(v-\sum_{e\in I} \ra(e))=t(v)-\sum_{e\in I}t(\ra(e)).$

Let us show that (P2) implies (P1). Let $v, w$ and $p$
be as in (P1). We prove the claim by induction on the length of $p.$ If
$|p|=0,$ $v=w$ and (P1) clearly holds. Assuming (P1) for
paths of length $n,$ let us prove (P1) for a path $p=ep_1$ where
$\so(e)=v, \ra(e)=\so(p_1),$ $\ra(p_1)=w,$ and $|p_1|=n.$ Indeed, 
$t(v)\geq t(\ra(e))$ by (P2) and $t(\ra(e))\geq t(w)$ by the induction
hypothesis. Thus, $t(v)\geq t(w).$  
\end{proof}

Condition (F) below is clearly necessary for a trace $t:L_K(E)\to R$ to be
faithful for some $*$-ring $R.$ 
\begin{enumerate}
\item[(F)] $t(v)>0$ for all vertices $v.$ 
\end{enumerate}

Neither (P) is sufficient for positivity nor (P) and (F) are sufficient
for faithfulness of a trace on a Leavitt path algebra as it
was observed in
\cite[Example 30]{Zak_paper}. In this example, $\Cset[x,x^{-1}]$ was considered 
as the Leavitt path algebra of the single-vertex single-edge graph over $\Cset$.
With the complex-conjugate involution on $\Cset,$ the trace defined by
$t(x^n)=i^n$ for $n\geq 0$ and $t(x^n)=i^{-n}$ for $n<0$ is such that (P) and
(F) hold. However, by considering the trace of the positive element
$(1+x)(1+x^{-1})$ one can see that $t$ is not positive (thus also not 
faithful). 

The fact that (P) is not sufficient for positivity and (P) and (F) are
not sufficient for faithfulness of a trace should not be surprising since traces
are rather general classes of maps. We show that this drawback
is not present for a certain class of well-behaved traces. We refer to such
traces as {\em canonical traces}. This terminology will be justified in
Proposition \ref{canonical}. We define a canonical trace using the following
proposition.

\begin{proposition} Any map $t$ on $G_E=\{pq^*\,|\, p,q$ paths with
$\ra(p)=\ra(q)\}$ such that
$$t(pq^*)=\delta_{p,q}t(\ra(p))$$ uniquely extends to a $K$-linear trace on
$L_K(E).$

If $t$ is a trace on $L_K(E)$, the following conditions
are equivalent. 
\begin{enumerate}
\item $t(pq^*)=\delta_{p,q}t(\ra(p))$ for all paths $p$ and $q.$ 

\item $t(pq^*)=0$ for all paths $p$ and $q$ of
non-equal length.  
\end{enumerate} 
Conditions (1) and (2) imply condition (3) below. If $K$ has
characteristic zero, then the
conditions (1) and (2) are equivalent to (3).   
\begin{itemize}
\item[(3)] $t(pq^*)=k^{|p|-|q|}t(pq^*)$ for any nonzero $k\in K$.  
\end{itemize}
\label{gauge_invariant}
\end{proposition}
\begin{proof}
To show the first sentence, consider \cite[Proposition
19]{Zak_paper} proving that any map $\delta$ on $G_E\cup\{0\}$ which
preserves zero is
central if and only if the following three conditions hold:
\begin{enumerate}
\item[$(i)$] If $\delta (x) \neq 0$ for some $x \in G_E$, then either $x=pqp^*$
or $x=pq^*p^*$ for some path $p$ and some closed path $q$.
\item[$(ii)$] $\delta(pqp^*)=\delta(q)$ and $\delta(pq^*p^*)=\delta(q^*)$ for
any path $p$ and any closed path $q$.
\item[$(iii)$] $\delta (p) = \delta (q)$ and $\delta (p^*) = \delta (q^*)$ 
for any two closed paths $p$ and $q$ such that $xy=p$ and $yx=q$ for some paths
$x$ and $y.$  
\end{enumerate}
It is easy to check that the map $t$ as in the first sentence of the
proposition satisfies these three conditions. Thus 
$t$ is a central map on $G_E.$ Since every element of $L_K(E)$ is a $K$-linear
combination of elements from $G_E,$ the map $t$ extends to a $K$-linear trace
of $L_K(E).$ This extension is unique since if two $K$-linear maps agree on
$G_E$ then they are equal on $L_K(E).$ 

If $t$ is a trace on $L_K(E),$ let us show the equivalence of conditions (1) and
(2). (1) clearly implies (2).
Since $t(pp^*)=t(p^*p)=t(\ra(p)),$ to show the
converse it is sufficient to show that $t(pq^*)\neq 0$ and $|p|=|q|$ imply  
 $p=q.$ Let us use induction on the length $|p|=|q|$.
If $p$ and $q$ are vertices, the claim clearly holds by axiom (V). 

Assume that the claim holds for paths $p$ and $q$ with  $|p|=|q|=n$ and let us
prove the claim for paths $p=e p_1$ and $q=f q_1$ where $e$ and $f$ are edges
and $p_1$ and $q_1$ paths with $|p_1|=|q_1|=n$, $\ra(e)=\so(p_1)$ and
$\ra(f)=\so(q_1).$ Then 
$0\neq t(pq^*)=t(ep_1q_1^*f^*)=t(f^*ep_1q_1^*)$ implies that $f^*ep_1q_1^*\neq
0$ and so $f^*e\neq 0$ thus $e=f$. Then we can use the induction hypothesis
for $t(p_1q_1^*)=t(e^*ep_1q_1^*)=t(f^*ep_1q_1^*)=t(ep_1q_1^*f^*)\neq 0$ to
obtain that $p_1=q_1.$ Thus $p=ep_1=eq_1=q.$  

(1) implies (3). Indeed, if $k\in
K$ is nonzero and $p\neq q,$ then 
$t(pq^*)=z^{|p|-|q|}t(pq^*)$ trivially holds since both sides are zero by (1).
If $p=q,$ then  $t(pp^*)=z^{|p|-|p|}t(pp^*)$ also holds. 

Now let us assume that (3) holds and that char$K=0$ and let us show (2). Assume
that $|p|\neq |q|.$ Then
\[t(pq^*)=k^{|p|-|q|}t(pq^*)\mbox{ implies }
(1-k^{|p|-|q|})t(pq^*)=0\mbox{ for every }0\neq k\in K.\]
Since char$K=0,$ we can find a nonzero element $k$ in $K$ that is
not a $(|p|-|q|)$-th root of the identity in $K$ in case $|p|>|q|.$ If 
$|q|>|p|,$ consider $k^{-1}$ for $k$ that is not a $(|q|-|p|)$-th root of the
identity. In both cases, $1-k^{|p|-|q|}\neq 0$ and so $t(pq^*)= 0$.
\end{proof}

Condition (3) from Proposition \ref{gauge_invariant} is the algebraic
version
of the definition of a gauge invariant trace on a graph $C^*$-algebra. Namely,
the gauge action on a graph
$C^*$-algebra given as in \cite[Definition 2.13]{Abrams_Tomforde} generalizes
to Leavitt path algebras as follows. 

The {\em gauge action} on $L_K(E)$ is a group
homomorphism $\lambda : K\setminus\{0\} \to $ Aut $(L_K(E))$ such that
for any
$0\neq k\in
K,$ $\lambda(k)(v) =v$
for all vertices $v$,  $\lambda(k)(e)=ke $ and  $\lambda(k)(e^*)=\frac{1}{k}e^*
$ for all edges $e.$ 
It is easy to see that in this case  
\[\lambda(k)(pq^*)=k^{|p|-|q|}pq^*\]
for any paths $p$ and $q$. This fact and Proposition \ref{gauge_invariant}
motivate the following definition. 

\begin{definition}
If $t$ is a trace on $L_K(E)$ and $p$ and $q$ paths, then 
\begin{enumerate}
\item $t$ is {\em gauge invariant} if
\[t(pq^*)=k^{|p|-|q|}t(pq^*)\;\;\mbox{ for any nonzero }k\in K.\]  

\item $t$ is {\em canonical} if
\[t(pq^*)=\delta_{|p|,|q|}t(pq^*)=\delta_{p,q}t(pq^*)=\delta_{p,q}t(\ra(p)).\]  
\end{enumerate}
\label{definition_canonical}
\end{definition}

The equalities in (2) follow from Proposition \ref{gauge_invariant}.
Also by Proposition
\ref{gauge_invariant}, every canonical trace is gauge invariant and 
the converse is also true if $K$ has characteristic zero.  

We justify the use of the term canonical by Proposition \ref{canonical}
which shows that every canonical trace is uniquely determined by its value
on vertices. Note that some maps on vertices cannot be extended to   
traces because their values may be such that axiom (CK2) is
violated. For example,    
consider the graph $E$ below and the $\Cset$-valued map which maps $u$ and $w$
to 1 and $v$
to 3.
$$\xymatrix{\bullet^u & \ar[l]_e\bullet^v\ar[r]^f& \bullet^w} $$
This map satisfies conditions (P) and (F) but it cannot be extended to a
$\Cset$-valued trace $t$ on the Leavitt path algebra $L_{\Cset}(E)$ since
$3=t(v)=t(ee^*+ff^*)=t(e^*e)+t(f^*f)=t(u)+t(w)=2$ in that case.

This example illustrates that the values on vertices have to agree with (CK2).
This fact was noticed by Tomforde in
\cite{Tomforde_arxiv} and later also utilized in  \cite{Johnson} and
\cite{Pask-Rennie}. Tomforde considers maps $\delta$ on vertices with
values in
$(0,\infty)$ which satisfy the following
two conditions and calls them graph traces on $E.$
\begin{enumerate}
\item[$(1)$] For all regular vertices $v$ we have $\delta (v) = \sum
_{e\in
\so^{-1}(v)} \delta (\ra(e))$.
\item[$(2)$] For all infinite emitters $v$ and every $e_1,\ldots,e_n\in
\so^{-1}(v),$ 
$\delta (v) \geq \sum_{i=1}^n \delta (\ra(e))$.
\end{enumerate}
Since condition (2)
follows from condition (P), we define a graph trace by 
condition (1) only. We
also allow the values of a graph trace to be in any involutive $K$-algebra $R$
not necessarily real interval $(0, \infty).$ 

\begin{definition}
If $R$ is a ring, a map $\delta: E^0\to R$ is a
{\em graph trace on $E$} if  
\begin{enumerate}
\item[(CK2)$_\delta$] $\;\;\delta (v) = \sum_{e\in \so^{-1}(v)} \delta
(\ra(e))\;\;$ for all regular vertices $v$.
\end{enumerate} 

If $R$ is a $*$-ring, the graph trace $\delta$ is {\em
positive}
if 
\begin{enumerate}
\item[(P)$_\delta$] $\;\;\delta(v)\geq \sum_{e\in I}
\delta(\ra(e))\;\;$ for all vertices $v$ and  finite subsets $I$ of $\so^{-1}(v)$ 
\end{enumerate} 
where $\sum_{e\in \emptyset}
\delta(\ra(e))$ is defined to be $0$. 

A positive graph trace $\delta$ is {\em faithful} if 
\begin{enumerate}
\item[(F)$_\delta$] $\;\;\delta(v)>0$ for all vertices $v.$ 
\end{enumerate}
\end{definition}

\begin{proposition}
Let $R$ be a $K$-algebra. There is a bijective correspondence $\tau$
between 
\begin{enumerate}
\item $R$-valued graph traces on a graph $E$ and

\item canonical, $K$-linear, $R$-valued traces on $L_K(E).$
\end{enumerate}
\label{canonical}
\end{proposition}
\begin{proof}
Let $\delta$ be a graph trace on $E.$ Define the map $t_\delta$ on $G_E$
by $ t_\delta(pq^*)=\delta_{p,q}\delta(\ra(p))$ and let $t_\delta(0)=0$.
By \cite[Theorem 28]{Zak_paper}, $t_{\delta}$ is central and extends to a
$K$-linear, $R$-valued trace $t_\delta$ on $L_K(E)$. The trace $t_\delta$ is
canonical by construction and its restriction to vertices is $\delta$. 

Conversely, if $t$ is a canonical, $K$-linear, $R$-valued trace on $L_K(E),$
then the restriction $\delta$ of $t$ to $E^0$ is a graph trace by axioms
(CK1), (CK2)
and the fact that $t$ is central:  \[\delta(v) =
t\left(\sum_{e\in \so^{-1}(v)} ee^*\right) = \sum_{e\in \so^{-1}(v)} t(ee^*) =
\sum_{e\in \so^{-1}(v)} t(e^*e) = \sum_{e\in \so^{-1}(v)} t(\ra(e))=\sum_{e\in
\so^{-1}(v)} \delta(\ra(e))\]
for any regular vertex $v.$ Then
$t(pq^*)=\delta_{p,q}\delta(\ra(p))=t_\delta(pq^*)$ for all paths $p$ and $q$
since $t$ is canonical. Thus $t=t_\delta$ on $G_E$ and, consequently,
$t=t_\delta$ on $L_K(E).$ 
\end{proof}

\section{Characterizations of positive and faithful canonical traces}
\label{section_positive_faithful}

By Proposition \ref{canonical}, a canonical, $K$-linear trace
on $L_K(E)$ can be seen as a well-behaved representative of all traces on
$L_K(E)$ with the same values on vertices. In this section, we prove similar
characterizations of canonical, $K$-linear traces which are positive (Theorem
\ref{positive}) and which are faithful (Theorem \ref{faithful}). As a
consequence, there is a bijective correspondence between positive, canonical,
$K$-linear traces and positive graph traces and a bijective correspondence
between faithful, canonical, $K$-linear traces with values in a positive
definite algebra and faithful graph traces (Theorem \ref{correpondence}). We
start by several lemmas.

\begin{lemma}
If $t$ is a canonical trace on $L_K(E)$ and $p,q,r,s$ any paths 
then $t(pq^*rs^*)\neq 0$ implies that either 
\begin{center}
{\em Case 1.} $s=pu$ and $r=qu$, or \hskip2cm {\em Case 2.} $p=su$ and $q=ru$ 
\end{center}
for some path $u.$ In both cases $t(pq^*rs^*)=t(uu^*)=t(\ra(u)).$
\label{fix1}
\end{lemma}
\begin{proof}
If $t(pq^*rs^*)\neq 0$ then $t(s^*pq^*r)\neq 0$ so  $s^*pq^*r\neq 0.$ This
implies that either  $s=pu$ or $p=su$  and $r=qv$ or  $q=rv$ for
some paths $u$ and $v.$

If  $s=pu$  and $r=qv$ then $pq^*rs^*=pq^*qvu^*p^*=pvu^*p^*\neq 0.$ So $0\neq
t(pvu^*p^*)=t(p^*pvu^*)=t(vu^*)$ implies that $u=v$ since $t$ is canonical. If 
$p=su$ and $q=rv,$ we obtain $u=v$ similarly. 

If $s=pu$ and $q=rv,$ then $pq^*rs^*=pv^*r^*ru^*p^*=pv^*u^*p^*\neq 0$. So
$0\neq t(pv^*u^*p^*)= t(p^*pv^*u^*)=t(v^*u^*)$ implies that $v^*u^*$ is a vertex
(necessarily $\ra(q)=\ra(s)$ in this case.) Thus $s=p$ and
$q=r$. Hence this case falls under the previous case with $u=v=\ra(p)=\ra(q).$ 
We reach a similar conclusion if $p=su$ and $r=qv.$ 
  
Thus, we have that either (Case 1) $s=pu$ and $r=qu$ so $rs^*=quu^*p^*$ in
which case $t(pq^*rs^*)=t(puu^*p^*)=t(uu^*)=t(\ra(u)),$ or (Case 2)
$p=su$ and $q=ru$ so  $pq^*=suu^*r^*$ in which case
$t(pq^*rs^*)=t(suu^*s^*)=t(uu^*)=t(\ra(u)).$
\end{proof}

Let us define a partial order $\preceq$ on the set $G_E=\{ pq^*\ | \ p$
and $q$ paths with $\ra(p)=\ra(q) \}$ by
\[puu^*q^* \preceq pq^*\;\;\mbox{ for any path }u\]
in which case we say that $puu^*q^*$  is {\em reducible} to
$pq^*$. 

We say that $pq^*$ is {\em irreducible} if 
$pq^*\preceq rs^*$ implies that $p=r$ and $q=s.$ If $pq^*$ is irreducible, 
$p=ru$ and $q=su$ only for paths $u$ of length zero. 

We refer to the $G_E$ elements of the form $pu_1u_1^*q^*$ and $pu_2u_2^*q^*$ as
{\em comparable} and we write $$pu_1u_1^*q^*\sim pu_2u_2^*q^*$$ in this case.
Such two comparable elements can both be reduced to $pq^*$. The following lemma
establishes that comparability is an equivalence relation and that elements of
a finite set of comparable elements can be reduced to the same {\em irreducible}
element. 

\begin{lemma} 
\begin{enumerate}
\item Every element of $G_E$ can be reduced to a unique irreducible element.

\item Two elements of $G_E$ are comparable if and only if they can be reduced
to the same irreducible element. Such irreducible element is unique. 

\item Relation $\sim$ is an equivalence relation on $G_E$ and elements from the
same equivalence class reduce to the same irreducible element, necessarily
unique.  

\item If $t(pq^*(rs^*)^*)\neq 0$  for some canonical trace $t$ on $L_K(E)$
and some $pq^*, rs^*\in G_E,$ then $pq^*$ and $rs^*$ are comparable. 
\end{enumerate}
\label{irreducible}
\end{lemma}
\begin{proof}
(1) Consider $pq^*\in G_E.$ If $pq^*$ is irreducible, we are done. If not,
$p=ru$ and $q=su$ for some path $u$ of nonzero length and $rs^*\in G_E$. If
$rs^*$ is irreducible, we are done. If not, repeat the argument for $rs^*$.
Since we are either shortening the length of paths in each step or we end up
with an irreducible element, the process ends after finitely many steps and we 
arrive to an irreducible upper bound of $pq^*.$ 

To show uniqueness, assume that there are
irreducible elements  $r_1s_1^*$ and $r_2s_2^*$ of $G_E$ and
paths $t_1, t_2$ such that $pq^*=r_1t_1t_1^*s_1^*$ and $pq^*=r_2t_2t_2^*s_2^*.$
These relations imply that $p=r_1t_1=r_2t_2$, $q=s_1t_1=s_2t_2.$ Thus we have
that either $r_1$ is a prefix of $r_2$ or vice versa and that $s_1$ is a prefix
of $s_2$ or vice versa. In any of these cases, we claim that $r_1=r_2$ and
$s_1=s_2.$

If $r_1$ is a prefix of $r_2$ and $s_1$ is a prefix of $s_2,$ then $r_2=r_1u_1$
and $s_2=s_1u_2$ for some paths $u_1$ and $u_2.$ In this case,
$r_1t_1=r_1u_1t_2$ and $s_1t_1=s_1u_2t_2$ and so $t_1=u_1t_2=u_2t_2$ which
implies $u_1=u_2.$ Then we have that $r_2s_2^*=r_1u_1u_1^*s_1^*\preceq r_1s_1^*$
and so $u_1$ has to be a path of length zero by irreducibility of $r_2s_2^*$.
Thus  $r_2=r_1u_1=r_1$ and $s_2=s_1u_1=s_1.$ The case when $r_2$ is a prefix of
$r_1$ and $s_2$ is a prefix of $s_1$ is handled similarly. 

If $r_1$ is a prefix of $r_2$ and $s_2$ is a prefix of $s_1,$ then $r_2=r_1u_1$
and $s_1=s_2u_2$ for some paths $u_1$ and $u_2.$ In this case,
$r_1t_1=r_1u_1t_2$ and $s_2u_2t_1=s_2t_2$ and so $t_1=u_1t_2$ and $u_2t_1=t_2$
which implies $u_1u_2t_1=t_1$ and $u_2u_1t_2=t_2.$ This implies that $u_1u_2$
and $u_2u_1$ are paths of length zero and so $u_1=u_2$ is a vertex. Thus,
$r_2=r_1u_1=r_1$ and $s_1=s_2u_2=s_2u_1=s_2.$ The case when $r_2$ is a prefix of
$r_1$ and $s_1$ is a prefix of $s_2$ is handled similarly. Since $r_1=r_2$ and
$s_1=s_2$ in any case, $pq^*$ reduces to a unique irreducible element.

(2) Let $p_1q_1^*$ and $p_2q_2^*$ be comparable elements of $G_E$.  We claim
that they reduce to the same irreducible element. Since $p_1q_1^*$ and
$p_2q_2^*$ are comparable,  $p_1q_1^*=pu_1u_1^*q^*$ and 
$p_2q_2^*=pu_2u_2^*q^*$ for some $pq^*\in G_E$ and paths $u_1, u_2.$ Let $rs^*$
be the irreducible element to which $pq^*$ reduces. Thus $pq^*=ruu^*s^*$ for
some path $u$. Thus  $p_1q_1^*=ruu_1u_1^*u^*s^*$ and 
$p_2q_2^*=ruu_2u_2^*u^*s^*$ and so $p_1q_1^*$ and $p_2q_2^*$ reduce to
irreducible $rs^*.$ The uniqueness of $rs^*$ follows from part (1). 

The converse follows by the definition of comparability. 

(3) Relation $\sim$ is clearly reflexive and symmetric. To
show transitivity, let $p_iq_i^*\in G_E, i=1,2,3,$    
$p_1q_1^*\sim p_2q_2^*,$ and $p_2q_2^*\sim p_3q_3^*$. By part (2), there are
irreducible elements  $pq^*$ and $rs^*$ of $G_E$ such that 
$p_1q_1^*$ and $p_2q_2^*$ reduce to $pq^*$ and $p_2q_2^*$ and $p_3q_3^*$ reduce
to $rs^*.$ Then $p_2q_2^*$ reduces to both $pq^*$ and $rs^*.$ By (1), $p=r$ and
$q=s$ and all three elements reduce to $pq^*$. Thus, 
$p_1q_1^*\sim p_3q_3^*$ by part (2). 

The second part of the claim follows from the transitivity of $\sim$ and parts (1)
and (2). 

(4) If $t(pq^*(rs^*)^*)=t(pq^*sr^*)\neq 0$  for some canonical trace $t$ on
$L_K(E)$ and some $pq^*, rs^*\in G_E,$ then either  $r=pu$ and $s=qu$ or $p=ru$
and $q=su$ for some path $u$ by Lemma \ref{fix1}. In the first case
$rs^*=puu^*q^*$ and in the second case $pq^*=ruu^*s^*$. In both cases
$pq^*$ and $rs^*$ are comparable. 
\end{proof}

The following lemma is the last one we need for the proof of Theorem
\ref{positive}. 

\begin{lemma} Let $t$ be a canonical, $K$-linear
trace on $L_K(E)$.
\begin{enumerate}
\item If $a_i\in K,$ and $r_ir_i^*, pr_ir_i^*q^*\in G_E$ for $i=1,\ldots, m,$ 
\[
x  =  \sum_{i=1}^m a_i pr_ir_i^*q^*\;\;\mbox{ and }\;\;
y  =  \sum_{i=1}^m a_i\,r_ir_i^*,
\]
then
\[t(xx^*)=t(yy^*).\]

\item If $t$ satisfies condition (P) from Lemma
\ref{condition_P},  
\[
x  =  \sum_{i=1}^m \sum_{j=1}^{m_i} a_{ij} e_{i}r_{ij}r_{ij}^*
e_i^*+av\;\;\mbox{ and }\;\;
y  =  \sum_{i=1}^m \sum_{j=1}^{m_i} a_{ij} e_{i}r_{ij}r_{ij}^* e_i^* +
a\sum_{i=1}^m e_ie_i^*
\]
where $v$ is a vertex of $E$, $\{e_1,\ldots, e_m\}\subseteq \so^{-1}(v),$ 
$e_i\neq e_j$ for $i\neq j,$ $a, a_{ij}\in K,$ $e_{i}r_{ij}r_{ij}^* e_i^*\in
G_E$ for $i=1,\ldots, m$ and $j=1,\ldots, m_i,$ then 
\[t(xx^*)\geq t(yy^*).\]
\end{enumerate}
\label{induction_tool}
\end{lemma}
\begin{proof}
To prove (1), note that $$xx^*=
\sum_{i=1}^m\sum_{j=1}^m a_ia_j^*  pr_ir_i^*q^*qr_jr_j^*p^*=
\sum_{i=1}^m\sum_{j=1}^m a_ia_j^*  pr_ir_i^*r_jr_j^*p^*\mbox{ 
and }yy^*=\sum_{i=1}^m\sum_{j=1}^m a_ia_j^* r_ir_i^*r_jr_j^*.$$
Thus,
\[t(xx^*) 
=  \sum_{i=1}^m\sum_{j=1}^m a_ia_j^* t(pr_ir_i^*r_jr_j^*p^*)
=  \sum_{i=1}^m\sum_{j=1}^m a_ia_j^*  t(p^*pr_ir_i^*r_jr_j^*)
=  \sum_{i=1}^m\sum_{j=1}^m a_ia_j^*  t(r_ir_i^*r_jr_j^*)
\]
This last expression is equal to $t(yy^*).$

To prove (2), compute that
\[
xx^* =\underline{
\sum_{i=1}^m \sum_{j=1}^{m_i}\sum_{k=1}^{m_i} a_{ij} a_{ik}^*
e_{i}r_{ij}r_{ij}^*r_{ik}r_{ik}^* e_i^*+ \sum_{i=1}^m \sum_{j=1}^{m_i}
(a_{ij}a^*+aa_{ij}^*) e_{i}r_{ij}r_{ij}^* e_i^*}
+ aa^*v\]
by using the fact that $e_i\neq e_j$ for $i\neq j$ and that $\so(e_i)=v$ for
all $i,j=1,\ldots, m.$ Similarly,                                   
\[
yy^* = 
\underline{
\sum_{i=1}^m \sum_{j=1}^{m_i}\sum_{k=1}^{m_i} a_{ij}a_{ik}^*
e_{i}r_{ij}r_{ij}^*r_{ik}r_{ik}^* e_i^* +
\sum_{i=1}^m \sum_{j=1}^{m_i} (a_{ij}a^*+aa_{ij}^*) e_{i}r_{ij}r_{ij}^* e_i^*
} + aa^*\sum_{i=1}^m e_ie_i^*.
\]
Since the underlined parts are equal, 
\[t(xx^*-yy^*)=
t(aa^*v-aa^*\sum_{i=1}^m e_ie_i^*)=
aa^*\left(t(v)-\sum_{i=1}^m t(e_ie_i^*)\right)=\]
\[=
aa^*\left(t(v)-\sum_{i=1}^m t(\ra(e_i))\right)=
aa^*t(v-\sum_{i=1}^m \ra(e_i))
\geq 0\]
by (P). Thus $t(xx^*) \geq t(yy^*).$
\end{proof}

We can now prove the main results of this section starting with the following. 

\begin{theorem}
Let $t$ be a canonical,
$K$-linear trace on $L_K(E).$ Then $t$ is positive if and only if  
\begin{enumerate}
\item[(P)] $\;\;\;t(v-\sum_{e\in I} \ra(e))\geq 0\;\;\;$ for 
all vertices $v$ and finite subsets $I$ of $\so^{-1}(v)$. 
\end{enumerate}
\label{positive}
\end{theorem}
\begin{proof}
If $t$ is positive, then condition (P) holds by 
\cite[Proposition 29]{Zak_paper} and Lemma \ref{condition_P}. 

To prove the converse, it is sufficient to show that $t(xx^*)\geq 0$ for every 
$x\in L_K(E).$ 
The elements of $G_E$ generate $L_K(E)$ as a $K$-algebra so it is sufficient to
assume that $x$ is a $K$-linear combination of $G_E$ elements. Such $x$ can
be written as 
\begin{equation}\label{form1}
x=\sum_{i=1}^n x_i
\end{equation}
where the elements
$x_i$ are $K$-linear combinations of comparable elements reducible
to the same irreducible element $p_iq_i^*\in G_E$ and all irreducible
elements $p_iq_i^*$, $p_jq_j^*$ are different, thus not comparable, for $i\neq
j$. This representation
of $x$ is possible by part (3) of Lemma \ref{irreducible}.
Since $t(x_ix_j^*)=0$ for $i\neq j$ by construction and part (4) of Lemma
\ref{irreducible}, 
$t(xx^*)=\sum_{i=1}^n t(x_ix_i^*)$.

Thus, it is sufficient to consider elements $x$ which are $K$-linear
combinations of $G_E$ elements comparable to each other. Let $x$ be one such
element. By part (3) of Lemma \ref{irreducible}, there is an irreducible $G_E$
element $pq^*$ such that $x$ can be written as
\begin{equation}\label{form2}
x=\sum_{j=1}^l a_j pr_jr_j^*q^*
\end{equation}
where $pr_jr_j^*q^*\in G_E$ and $a_j\in K$ for $j=1,\ldots, l.$
By part (1) of Lemma \ref{induction_tool}, $t(xx^*)=t(yy^*)$ where 
$y=\sum_{i=j}^l a_jr_jr_j^*.$ Note that
all paths $r_j$ have the same source $\ra(p)$ since $pr_j\neq 0$.
Thus, it is sufficient to consider elements $x$ of the form
\begin{equation}\label{form3}
x=\sum_{j=1}^l a_jr_jr_j^*,\;\;\; \so(r_i)=\so(r_j)\mbox{ for all
}i,j=1,\ldots, l.
\end{equation}

Let $v$ denote the source of all $r_j, j=1,\ldots,l.$ If several different
paths $r_j$ have zero length, group them in a single term by writing
$av+bv$ as $(a+b)v$ for $a,b\in K$. Depending on the coefficient with $v$
being zero or nonzero, we have three possible cases. 
\begin{enumerate}
\item[Case 1.] None of the paths $r_j$ have zero length.

\item[Case 2.] Exactly one of the paths $r_j$ has zero length and $l=1$.

\item[Case 3.] Exactly one of the paths $r_j$ has zero length and $l>1.$
\end{enumerate}

In case 1, let $e_i, i=1,\ldots, m$ be
the list of all edges that are the first in paths $r_j, j=1,\ldots, l$ without
repetition. Let us denote every $r_j$ as $e_ir_{ik}$ 
for some paths $r_{ik}$ in which case we write $a_j$ as $a_{ik}.$ Then we can
write $x$ as
\begin{equation}\label{form4}
x=\sum_{i=1}^m x_i\;\;\;\mbox{ for }\;\;\;
x_i=\sum_{k=1}^{m_i} a_{ik} e_{i}r_{ik}r_{ik}^*e_i^*
\end{equation}
and $l=\sum_{i=1}^m m_i$ necessarily. If $i\neq j,$
$e_i^*e_j=0$ and so $x_ix_j^*=0.$ Thus  
$t(xx^*)=\sum_{i=1}^m t(x_ix_i^*).$

For every $x_i, i=1,\ldots m,$ we can apply part (1) of Lemma
\ref{induction_tool} to obtain that $t(x_ix_i^*)=t(z_iz_i^*)$ where 
\begin{equation}\label{form5}
z_i=\sum_{k=1}^{m_i} a_{ik} r_{ik}r_{ik}^*.
\end{equation}

In case 2, $l=1$ and $x$ has the form $x=a_1v$. Since $t(v)$ is positive by
condition (P), $t(xx^*)=a_1a_1^*t(v)$ is positive as well. 

In case 3, we show that the consideration reduces to either case 1 or case
2. Since there is one $r_j$ with zero length, we can assume it is $r_l.$ Let
$e_i, i=1,\ldots, m$ be the list of all edges that are the first in paths $r_j,
j=1,\ldots, l-1$ without repetition. Let $r_j=e_ir_{ik}$ for some paths $r_{ik}$
and let us represent $a_j$ as $a_{ik}$  so that we can write $x$ as
\[x=\sum_{i=1}^m \sum_{k=1}^{m_i} a_{ik} e_{i}r_{ik}r_{ik}^* e_i^*+a_lv.\]
Here $l-1$ is necessarily equal to $\sum_{i=1}^m m_i.$ In this case,
$t(xx^*)\geq t(yy^*)$ where
\[y=\sum_{i=1}^m \sum_{k=1}^{m_i} a_{ik} e_{i}r_{ik}r_{ik}^*
e_i^*+a_l\sum_{i=1}^me_ie_i^*\]
by part (2) of Lemma \ref{induction_tool}. So, it is sufficient to show that the
trace of $yy^*$ is positive. Since we can regroup the terms of $y$ so that
\[y=\sum_{i=1}^m y_i\;\;\;\mbox{ for }\;\;\;
y_i=\sum_{k=1}^{m_i} a_{ik} e_{i}r_{ik}r_{ik}^*e_i^*+a_le_ie_i^*,\] 
the element $y$ falls under case 1 and is represented as in
(\ref{form4}). In this case the elements $z_i=\sum_{k=1}^{m_i}
a_{ik}r_{ik}r_{ik}^*+a_l\ra(e_i)$ are as in  (\ref{form5}) and  
$t(y_iy_i^*)=t(z_iz_i^*)$. Thus, we either reduce our consideration to
$z_i$ as in (\ref{form5}) of case 1 or the elements $z_i$ already have the form
as in case 2. 

All paths $r_j$ of nonzero length in formula (\ref{form3}) have strictly longer
length than the paths $r_{ik}$ in (\ref{form5}) since $r_j=e_ir_{ik}$.
The expression in formula
(\ref{form5}) can be written as in (\ref{form1}) and the whole process can be
repeated treating each $z_i$ as the original $x$ in formula (\ref{form1}). This
process terminates in finitely many steps and
eventually reduces the consideration of all the elements $x$ to those of the
form $av$ where $a\in K$ and $v\in E^0$. This situation
has been handled in case 2 above. So, this finishes the proof. 
\end{proof}

We turn to conditions characterizing faithfulness of a canonical
trace now. 

\begin{theorem} Let $R$ be a positive definite $K$-algebra and $t$ a canonical,
$K$-linear, $R$-valued trace on $L_K(E).$ Then $t$ is faithful
if and only if conditions (P) and (F) hold where 
\begin{enumerate}
\item[(F)]  $t(v)> 0$ for all vertices $v.$
\end{enumerate}
\label{faithful}
\end{theorem}
\begin{proof}
If $t$ is faithful then $t$ is positive so (P) holds. Condition (F) clearly
holds as well.                              

Assume now (P) and (F). By Theorem \ref{positive}, the trace $t$ is positive.
Since $R$ is positive definite, to show that $t$ is faithful  it is sufficient
to show
that $t(xx^*)= 0$ implies that $x=0$ for any $x\in L_K(E)$ by Lemma
\ref{positively_faithful}. The proof 
follows the stages of the proof of Theorem \ref{positive}, so the labels of the
formulas and notation refer to those used in the proof of Theorem
\ref{positive}. Writing
$x$ as
$\sum_{i=1}^n x_i$ as in formula (\ref{form1}), we have that 
$t(xx^*)=\sum_{i=1}^n t(x_ix_i^*)=0$. Since $t(x_ix_i^*)\geq 0$ because $t$
is positive, we have that $t(x_ix_i^*)= 0$ for
every $i=1,\ldots, n$ by the assumption that $R$ is positive definite.
Thus, it is sufficient to prove that each $x_i$ is zero.
So, it is sufficient to consider $x$ which has the form as in formula
(\ref{form2}).
For such $x$, $t(xx^*)=t(yy^*)$ where $y$ is as in (\ref{form3}). Since
$x=pyq^*$, to show that $x$ is zero, it is sufficient to show that $y$ is
zero. So, it is
sufficient to consider $x$ to be as in (\ref{form3}). 

For an element $x=\sum_{j=1}^l a_jr_jr_j^*$ with $\so(r_j)=v$
for all $j=1,\ldots, l$ as in (\ref{form3}), consider again cases 1, 2, and 3 as
in the proof of Theorem \ref{positive}. In case 1, write $x$ as in
(\ref{form4}). The terms $x_i, i=1,\ldots,m$ are such that
$0=t(xx^*)=\sum_{i=1}^m t(x_ix_i^*).$ Since $R$ is
positive definite and $t(x_ix_i^*)\geq 0$, we have that $t(x_ix_i^*)=0$ for all
$i.$ Thus, to show that
$x=0$ it is sufficient to show that $x_i=0$ for every $i.$ If  $z_i,
i=1,\ldots, m,$ are as in (\ref{form5}), then $x_i=e_iz_ie_i^*$. Thus, to show
that $x_i=0$ for all $i$, it is sufficient to show that $z_i=0$ for all $i.$

In case 2,  $x=a_1v$ where
$a_1\in K$ and
$v\in E^0.$ If $a_1\neq 0$ then $0=t(a_1a_1^*v)=a_1a_1^*t(v)$
implies
that $t(v)=0$ which contradicts (F). Thus $a_1=0$ and so $x=a_1v=0.$ 

In case 3, write $x$ as
\[x=\sum_{i=1}^m \sum_{k=1}^{m_i} a_{ik} e_{i}r_{ik}r_{ik}^* e_i^*+a_lv\]
where $e_i, i=1,\ldots, m$ 
is the list of all edges that are the first in paths $r_j, j=1,\ldots, l-1$
without repetition and $a_l\neq 0$ (otherwise $x$ falls under case 1). Also
as in the proof of Theorem \ref{positive}, let
\[y=\sum_{i=1}^m y_i\;\;\;\mbox{ for }\;\;\;
y_i=\sum_{k=1}^{m_i} a_{ik} e_{i}r_{ik}r_{ik}^*
e_i^*+a_le_ie_i^*\] and note that 
$0=t(xx^*)\geq t(yy^*)=\sum_{i=1}^m t(y_iy_i^*)\geq 0$ 
by part (2) of Lemma \ref{induction_tool} and so $\sum_{i=1}^m t(y_iy_i^*)=
0$. Since $R$ is positive definite and $t(y_iy_i^*)\geq 0,$ we have that
$t(y_iy_i^*)=0$ for all
$i=1,\ldots, m.$ By part (1) of Lemma \ref{induction_tool}, the elements 
\[z_i=\sum_{k=1}^{m_i} a_{ik} r_{ik}r_{ik}^*
+a_l\ra(e_i)\] are such that $0=t(y_iy_i^*)=t(z_iz_i^*)$. We claim that
showing $z_i=0$ for all $i=1,\ldots, m$ is sufficient to show that
$x=0.$ Indeed, if $z_i=0,$ then $y_i=e_iz_ie_i^*=0$ as well and so 
$\sum_{k=1}^{m_i} a_{ik} e_{i}r_{ik}r_{ik}^*e_i^*
=-a_le_ie_i^*.$ Thus, 
\[x=a_lv-\sum_{i=1}^m 
a_le_ie_i^*=a_l(v-\sum_{i=1}^m e_ie^*_i)\;\; \Rightarrow\;\;
xx^*=a_la_l^*(v-\sum_{i=1}^m e_ie^*_i).\]
Since $t(xx^*)=0$, we have that $a_la_l^*t(v-\sum_{i=1}^m e_ie^*_i)=0.$ This
implies  $t(v-\sum_{i=1}^m e_ie^*_i)=0$ because $a_l\neq 0.$ 
Note that $v$ is not a sink since $m>0.$ If  
$\{e_1, \ldots,e_m\}=\so^{-1}(v)$ then $v$ is regular,
$v=\sum_{i=1}^m e_ie^*_i$ by (CK2) and so $x=0$. If $\{e_1,
\ldots, e_m\}\subsetneq \so^{-1}(v),$ there is an
edge $f$ different from $e_1,\ldots, e_m$ with $v=\so(f).$ The element  $v-
ff^*-\sum_{i=1}^m e_ie^*_i$ is a projection (selfadjoint idempotent) and so it
is positive. Since $t$ is positive, $v\geq
ff^*+\sum_{i=1}^m e_ie^*_i$ implies that 
$t(v)\geq t(ff^*)+\sum_{i=1}^mt(ee^*).$ Since (F)
holds, $t(ff^*)=t(\ra(f))>0$ and so  
$$t(v)\geq t(ff^*)+\sum_{i=1}^mt(e_ie_i^*)>\sum_{i=1}^mt(e_ie_i^*).$$ 
This contradicts $0=t(v-\sum_{i=1}^m e_ie^*_i)=t(v)-\sum_{i=1}^m t(e_ie^*_i).$  

Thus, both case 1 and case 3 reduce to the consideration of 
elements $z_i$ as in (\ref{form5}). The expression in (\ref{form5}) can be
written as an
expression in (\ref{form1}) again and the whole process
can be repeated. In every step the lengths of the paths in formula
(\ref{form5}) are shorter than the lengths of the corresponding paths in
(\ref{form3}). Thus the process terminates in
finitely many steps and eventually reduces to the consideration of the
elements of the form $av$ where $a\in K$ and $v\in E^0,$ handled in case 2.  
Thus, we have that $t(xx^*)=0$ implies that $x=0$ for any $x.$
\end{proof}

Note that under assumptions of Theorem \ref{faithful}, $K$ is positive
definite by Proposition \ref{positively_faithful_on_LPA} and then so is $L_K(E)$
by \cite[Proposition 2.4]{GonzaloRangaLia}.

Theorems \ref{positive} and \ref{faithful} imply the following result.

\begin{theorem} Let $R$ be an involutive $K$-algebra. The correspondence $\tau$
from Proposition \ref{canonical} is
such that it induces a bijective correspondence between 
\begin{enumerate}
\item positive, $R$-valued graph traces on $E$ and

\item positive, canonical, $K$-linear, $R$-valued traces on $L_K(E).$
\end{enumerate}
If $R$ is positive definite, the correspondence $\tau$ is
such that it induces a bijective correspondence between 
\begin{enumerate}
\item faithful, $R$-valued graph traces on $E$ and

\item faithful, canonical, $K$-linear, $R$-valued traces on $L_K(E).$
\end{enumerate}
\label{correpondence}
\end{theorem}
\begin{proof}
Let us recall that  $t_\delta$ denotes                          
the unique extension of an $R$-valued graph trace $\delta$ on $E$ to a
canonical, $K$-linear, $R$-valued trace on $L_K(E)$ in the proof of Proposition
\ref{canonical}. 
Theorem \ref{positive} implies that if such graph trace $\delta$ is
positive, then $t_\delta$ is
positive as well. Conversely, if $t$ is a positive, canonical, $K$-linear,
$R$-valued trace, then
condition (P) holds. Since condition (P) implies (P)$_\delta$ for the
restriction of $t$ to vertices, the claim follows.  

Similarly, if $R$ is positive definite and a graph trace $\delta$ is faithful,
then $t_\delta$ is
faithful as well by Theorem \ref{faithful}. The converse clearly
follows since condition (F) implies (F)$_\delta$ for the
restriction of $t$ to vertices.
\end{proof}

The assumption that $R$ is positive definite is necessary in Theorem
\ref{faithful} and the second part of Theorem \ref{correpondence}. \cite[Example
34]{Zak_paper} can be used to demonstrate this.
In this example, the Leavitt path algebra of the graph with two vertices $v$ and
$w$
and one edge $e$ from $v$ to $w$ is considered over the field of complex
numbers with the identity involution. Mapping both vertices to
1 defines a faithful, $\Cset$-valued graph trace $\delta$ on $E$. Since $\delta$
is positive, it extends to a positive, canonical, $\Cset$-linear,
$\Cset$-valued
trace $t$ by Theorem \ref{correpondence} and $t$ is such that conditions (P)
and (F) are fulfilled. However, $t$ is not faithful since
$v-w=(v+iw)(v+iw)^*\geq 0,$ $t(v-w)=1-1=0$ and $v-w\neq 0.$

Theorem \ref{correpondence} has the following corollary.  
\begin{corollary}
The following conditions are equivalent for any positive definite field $K$. 
\begin{enumerate}
\item There is a faithful, canonical, $K$-linear, $K$-valued trace on $L_K(E).$

\item There is a faithful, canonical, $K$-valued trace on $L_K(E).$

\item There is a faithful, $K$-linear, $K$-valued trace on $L_K(E).$

\item There is a faithful, $K$-valued trace on $L_K(E).$

\item There is a faithful, $K$-valued graph trace on $E$.
\end{enumerate}
\end{corollary}
\begin{proof}
The implications $(1)\Rightarrow (2)\Rightarrow (4)$ and  $(1)\Rightarrow
(3)\Rightarrow (4)$ are tautologies. Condition (4) implies (5) since the
restriction of a trace as in (4) to $E^0$ is a graph trace as in (5). Finally,
(5)
implies (1) since every graph trace as in (5) extends to a trace as in (1) by
Theorem \ref{correpondence}.
\end{proof}

If any of the equivalent conditions (1)--(5) hold for a positive definite field
$K$, we say that $L_K(E)$ {\em admits a faithful trace}.

We conclude this section with another corollary of Theorem
\ref{correpondence}. Namely, in \cite[Proposition 3.9]{Pask-Rennie}, it has
been shown that there is a
bijective correspondence 
between faithful, $\Cset$-valued graph traces on a countable, row-finite
graph $E$ and faithful,
semifinite, lower semicontinuous, gauge invariant, $\Cset$-valued traces on
$C^*(E)$. In \cite{Pask-Rennie}, a trace $t$ on a $C^*$-algebra $A$ is defined
as an
additive map on the positive cone $A^+$ of $A$ taking values in $[0,\infty]$
such that
$t(ax)=at(x)$ for a nonnegative real number $a$ and $x\in A^+$ and 
$t(xx^*)=t(x^*x)$ for all $x\in A$. 
To avoid confusion
with our definition of a trace, we shall refer to such map as a {\em
$C^*$-trace}. Clearly every positive, $\Cset$-linear trace on $A$ is a
$C^*$-trace.

Recall that every element $x=a+ib$ of
a $C^*$-algebra $A$ where $a$ and $b$ are the real and imaginary parts of $x$
(see \cite[page 105]{KadRing}), can be written as a $\Cset$-linear combination
$x=(a^+-a^-)+i(b^+-b^-)$ where $a^+, a^-, b^+, b^-$ are positive elements such
that $a^+a^-=a^-a^+=b^+b^-=b^-b^+=0$
 and this representation is unique (\cite[Corollary 4.2.4]{KadRing}). Thus,
every $C^*$-trace defined on the positive cone of $A$ can be 
extended to $A$ by letting  
$t(x)=t(a^+)-t(a^-)+i(t(b^+)-t(b^-)).$ It is straightforward to check that
this extension is $\Cset$-linear and positive. Without any danger of confusion,
we shall refer to this extension of a $C^*$-trace as a $C^*$-trace as well. 

A $C^*$-trace $t$ on a $C^*$-algebra $A$ is defined to be {\em faithful} if
$t(xx^*)=0$ implies that $x=0.$ This condition is equivalent to the one we
use to define a faithful additive map by Lemma \ref{positively_faithful}
because the complex-conjugate involution is positive definite and every
$C^*$-algebra is proper.
A $C^*$-trace $t$ on $A$ is {\em semifinite} if the set of elements of $A^+$
with finite trace is norm dense in $A^+$. A $C^*$-trace $t$ on $A$ is {\em
lower semicontinuous} if $t(\lim_{n\to\infty}
a_n)\leq \liminf_{n\to\infty}t(a_n)$ for all norm convergent sequences $a_n$ in
$A^+$. 

If $\{S_e, p_v \mid e\in E^1, v\in E^0\}$ is a Cuntz-Krieger
$E$-family for a graph $C^*$-algebra $C^*(E)$ (see \cite{{Gauge_Inv_Thm_paper}},
\cite{Tomforde_arxiv}
or \cite{Pask-Rennie} for example), the gauge action $\lambda$ on the unit
sphere $S^1$ is given by
$\lambda_z(S_pS_q^*)=z^{|p|-|q|}S_pS_q^*$ for $z\in S^1$ and paths $p$ and $q.$
A $C^*$-trace $t$ on
$C^*(E)$ is {\em gauge invariant} if
$t(\lambda_z S_pS_q^*)=t(S_pS_q^*)$ for every complex number $z$ of unit
length. Since such $t$ is $\Cset$-linear, this condition is equivalent to 
\begin{enumerate}
\item[(GI)] $\;\;\;t(S_pS_q^*)=z^{|p|-|q|}t(S_pS_q^*)\mbox{ for all complex
numbers }z\mbox{ of unit length}$
\end{enumerate}
and  all paths $p$ and $q.$ 
Thus, if a $C^*$-trace is gauge invariant in the sense of Definition
\ref{definition_canonical}, then it is gauge invariant in this sense. Our next
result, Corollary \ref{c-star}, shows that the converse holds for 
faithful, semifinite, lower semicontinuous $C^*$-traces on
$C^*(E)$ if $E$ is countable. Corollary \ref{c-star} also shows that it is not
necessary to assume that  $E$ is row-finite in \cite[Proposition 
3.9]{Pask-Rennie}.

\begin{corollary} Let $E$ be a countable graph and consider $\Cset$ with the
complex-conjugate involution. The following sets are in bijective
correspondences. 
\begin{enumerate}
\item The set of faithful, $\Cset$-valued graph traces on $E$,

\item the set of faithful, gauge invariant, $\Cset$-linear, $\Cset$-valued
traces on $L_{\Cset}(E),$ and

\item the set of faithful, semifinite, lower semicontinuous, gauge
invariant $C^*$-traces on $C^*(E)$.
\end{enumerate}
A faithful, semifinite, lower semicontinuous
$C^*$-trace on $C^*(E)$
satisfies {\em (GI)} if and only if
it is gauge invariant (in the sense of Definition
\ref{definition_canonical}).
\label{c-star}
\end{corollary}
\begin{proof}
Since the complex-conjugate involution is positive definite,
the sets (1) and (2) are in a bijective correspondence by Theorem
\ref{correpondence}. 

In \cite{Pask-Rennie}, $E$ is assumed to be countable
and row-finite. By \cite[Lemma 3.2]{Pask-Rennie}, every semifinite
$C^*$-trace on $C^*(E)$ is such that the trace of an element of $L_{\Cset}(E)$
is finite but the assumption that $E$ is
row-finite is not used it the proof. Thus, every $C^*$-trace as in (3)
restricts to a graph trace as in (1) by the proof of \cite[Lemma
3.2]{Pask-Rennie}. 

Thus, it remains to show that every trace as in (2) extends to a
trace as in (3). The proof
of  \cite[Proposition 3.9]{Pask-Rennie} shows this claim 
for $E$ countable and row-finite. The assumption that $E$ is row-finite is used
only when invoking the Gauge Invariant Uniqueness Theorem for row-finite
graphs from \cite{Gauge_Inv_Thm_row_finite}. This theorem has been shown for
countable graphs in \cite[Theorem 2.1]{Gauge_Inv_Thm_paper} so we need to
require just that $E$ is countable. 
The proof of \cite[Proposition 3.9]{Pask-Rennie}
shows that $t$ satisfies condition (GI). However,
since $t$ is canonical on $L_{\Cset}(E)$ and $\Cset$ has
characteristic zero, $t$ is gauge invariant by Proposition
\ref{gauge_invariant}. 

To prove the last sentence of this corollary, it is sufficient to prove that a
faithful, semifinite, lower semicontinuous,
$C^*$-trace on $C^*(E)$ which satisfies (GI) is canonical. If $t$ is such
a trace, the restriction of $t$ on the vertices is a faithful graph
trace $\delta$. The extension of $\delta$ to $L_{\Cset}(E)$ is a canonical trace
whose extension to $C^*(E)$ is $t.$ Thus, $t$ is canonical. 
\end{proof}

\section{Cohn-Leavitt algebras and directly finite Leavitt path algebras}
\label{section_finite}

In this section, we characterize directly finite Leavitt path algebras
as exactly those Leavitt path algebras $L_K(E)$ for which $E$ is a no-exit graph
(Theorem \ref{no-exit}). The proof of this characterization involves
consideration of Cohn-Leavitt algebras, the algebraic counterparts of relative
graph $C^*$-algebras, for which we also formulate all our previous results.
Lastly, we compare the classes of
locally noetherian, directly finite and those Leavitt
path algebras which admit a faithful trace.  

Recall that a unital ring is {\em directly (or Dedekind) finite}
if $xy=1$ implies that $yx=1$ for all $x$ and $y$. The involutive version of
this definition is the following: a ring is {\em finite} if
$xx^*=1$ implies $x^*x=1$ for all $x.$ This terminology comes from
operator theory and should not be confused with rings having
finite cardinality. In the rest of the paper, when we refer to a $*$-ring or
a $*$-algebra being finite, we assume the finiteness in this sense.  

We adapt finiteness and direct finiteness to 
non-unital rings  with local units. Recall that a ring $R$ has local units
if for every finite set
$x_1,\ldots, x_n\in R $ there is an idempotent $u$ such that $x_iu=ux_i=x_i$
for all $i=1,\ldots, n$.

\begin{definition}
A ring with local units $R$ is said to be {\em directly finite} if for every
$x,y\in R$ and an idempotent element $u\in R$ such that $xu=ux=x$ and 
$yu=uy=y$, we have that 
\[xy=u\mbox{ implies }yx=u.\]

A $*$-ring with local units $R$ is said to be {\em finite} if for every
$x\in R$ and an idempotent $u\in R$ such that $xu=ux=x,$ we have that 
\[xx^*=u\mbox{ implies }x^*x=u.\] Condition $xx^*=u$ implies that  $u$ is a
projection (selfadjoint idempotent) since $u^*=(xx^*)^*=xx^*=u.$ Thus, 
$x^*u=ux^*=x^*$ as well. 
\end{definition} 

If $R$ is a unital, directly finite ring, then it is directly finite in
the locally-unital
sense as well. Indeed, assuming that $xy=u$ for an idempotent element
$u$ with $xu=ux=x$ and $yu=uy=y$ we have that
$(x+1-u)(y+1-u)=xy+1-u=1.$ This 
implies that $1=(y+1-u)(x+1-u)=yx+1-u$ and from this it follows that 
$yx=u.$ Similarly, if $R$ is a unital, finite $*$-ring, then it is finite in
the locally-unital sense as well.

The fact that the existence of a faithful trace on a (unital)
von Neumann algebra implies its finiteness is well known and widely used. The
arguments proving this
fact easily generalize to any unital $*$-ring with a faithful trace. We note
this fact for locally unital rings. In fact, as the next proposition shows,
a more general claim holds: any ring with local units and
a trace which is injective on idempotents is directly finite.  

\begin{proposition}
If $R$ is a ring with local units and there is a trace on
$R$ which is injective on idempotent elements, then $R$ is directly finite. 

If $R$ is a $*$-ring with local units and there is a trace
on $R$ which is injective on projections, then $R$ is finite. In particular, a
$*$-ring with local units and a faithful trace is finite.   
\label{faithful_implies_finite}
\end{proposition}
\begin{proof}
Let $R$ be a ring with local units, $t$ a trace on
$R$ which is injective on idempotents, and let $x,y$ be in $R$ such that
$xu=ux=x$ and $yu=uy=y$ for some idempotent $u.$ If $xy=u,$ then $u-yx$ is an
idempotent since $(u-yx)(u-yx)=u-yx-yx+yxyx=u-yx-yx+yux=u-yx.$ 
Then 
$t(u-yx)=t(xy-yx)=0$ which implies that $u-yx=0$ since $t$ is
injective on
idempotents. Thus $yx=u.$

The second sentence is proven analogously and the third is a consequence of
the second. 
\end{proof}

Note that $L_K(E)$ is a ring with local units. Indeed for any $x_i$ in $L_K(E),$
$i=1,\ldots, n$
which can be represented using paths $p_{ij}, q_{ij}$ and $a_{ij}\in K,
j=1,\ldots, n_i$ as
$x_i=\sum_{j=1}^{n_i} a_{ij} p_{ij}q_{ij}^*,$ we have that the
sum $u$ of all vertices that are sources of all paths $p_{ij}$ and $q_{ij}$ for
all
$i=1,\ldots, n$ and $j=1,\ldots, n_i$ is an idempotent with $x_iu=x_i=ux_i.$ 

The direct finiteness of a Leavitt path algebra forces the underlying graph to
be no-exit. This has been shown to hold in \cite[Proposition
3.1]{GonzaloLia} for Leavitt path algebras of finite graphs. The proof of
part (6) of \cite[Proposition 29]{Zak_paper} shows this claim for any Leavitt
path algebra but since \cite[Proposition 29 (6)]{Zak_paper} is worded in a
different set up, we list the proof below. 
\begin{proposition}
If $L_K(E)$ is (directly) finite, then $E$ is no-exit.
\label{finite_implies_no-exit}
\end{proposition}
\begin{proof}
Since direct finiteness implies finiteness, it is sufficient to show
the claim assuming that $L_K(E)$ is finite. In this case, assume that $E$ has a
cycle $p$ with an exit
$e.$ We also may assume that $\so(p)=\ra(p)=\so(e)$, and we denote this
vertex by $v$. Let $w =\ra(e),$ $x = p+(1-\delta_{v,w})w$ and
$u=v+(1-\delta_{v,w})w.$ Then we have that $xu=ux=x$ and that 
$x^*x  = p^*p + (1-\delta_{v,w})w = v+(1-\delta_{v,w})w=u$.
By finiteness,
we then have that $ v+(1-\delta_{v,w})w=u = xx^* = pp^*+(1-\delta_{v,w})w$.
Hence, $v = pp^*$. But then $0 = e^*pp^* = e^*v=e^*$ which is a
contradiction. Thus, $p$ cannot have an exit.
\end{proof}

Our goal is to prove that the converse of Proposition
\ref{finite_implies_no-exit} holds. This has been proven for Leavitt path
algebras of finite graphs in \cite[Theorem 3.3]{GonzaloLia}. Thus, if we can 
``localize'' our main claim, i.e. reduce the consideration of the general
case to a Leavitt path algebra of a finite subgraph and then use \cite[Theorem
3.3]{GonzaloLia}, then we would achieve our goal. In particular, assuming that a
graph $E$ is no-exit and 
considering $x,y\in L_K(E)$ such that $xy=u$ for some local
unit $u$, we aim to consider a finite subgraph
$F$ generated by the vertices and edges of just those paths that appear in
representations of $x,y$ and $u$. The problem is that the subgraph $F$ defined
in this way may not be complete in the sense of \cite[Definition
9.7]{Abrams_Tomforde} and so $L_K(F)$ may not be a subalgebra of $L_K(E).$
However, we show that this impediment can be avoided by considering Cohn-Leavitt
algebras of \cite{Ara_Goodearl}. Namely, we can consider appropriate finite
subgraph $F$ such that the Cohn-Leavitt algebra of $F$ is a subalgebra of
$L_K(E)$ and we can adapt \cite[Theorem 3.3]{GonzaloLia} to Cohn-Leavitt
algebras of finite graphs. This approach requires us to recall the definition of
Cohn-Leavitt algebras and demonstrate some preliminaries. 

Cohn-Leavitt algebras are obtained by requiring the (CK2) axiom to hold just for
a portion of regular vertices, not necessarily all of them. More precisely, if
$S$ is a subset of regular
vertices, the {\em Cohn-Leavitt algebra $CL_K(E,S)$ of $E$ and $S$ over $K$}  is
a free $K$-algebra
generated by the sets $E^0\cup E^1\cup \{e^*\ |\ e\in E^1\}$
with relations (V), (E1), (E2), (CK1) and
\begin{itemize}
\item[(SCK2)] $v=\sum_{e\in \so^{-1}(v)} ee^*,$ for every vertex $v\in S$.
\end{itemize}
For the rest of the paper, $R(E)$ denotes the set of the regular vertices of $E$
and $S$ a subset of $R(E).$ 
If $S$ is empty, the Cohn-Leavitt algebra $CL_K(E,S)$ is a Cohn path algebra and
we write
$CL_K(E, \emptyset)$ as $C_K(E).$  If
$S$ is equal to $R(E)$ then $CL_K(E,S)$ is a Leavitt path algebra and we write
$CL_K(E, R(E))$ as $L_K(E).$   

The $C^*$-analog of Cohn-Leavitt algebras preceded the consideration of
Cohn-Leavitt
algebras. In \cite{Muhly_Tomforde}, Muhly and Tomforde introduced the {\em
relative graph $C^*$-algebra} $C^*(E,S)$ of a graph $E$ and $S\subseteq R(E)$ as
the $C^*$-algebra generated by a universal Cuntz-Krieger $(E,S)$-family,  
obtained by replacing the (CK2) axiom of a Cuntz-Krieger $E$-family 
by the (SCK2) axiom (\cite[Definition 3.5]{Muhly_Tomforde}). In
\cite{Ara_Goodearl}, Cohn-Leavitt algebras are introduced for a more general
class of graphs, called separated graphs, than those considered in this paper.
The graphs considered in this paper correspond to those from \cite{Ara_Goodearl}
with trivial separation.

If $E$ is a countable graph, \cite[Theorem 3.7]{Muhly_Tomforde} shows that the
relative graph $C^*$-algebra $C^*(E,S)$ is canonically isomorphic to the graph
$C^*$-algebra $C^*(E_S)$ of a suitable graph $E_S$. In the paragraph preceding
Lemma \ref{star_iso}, we review this construction and adapt it to show that
$CL_K(E, S)$ is isomorphic to $L_K(E_S)$ for any graph $E$. Thus, the class of
Cohn-Leavitt algebras is not larger than the class of Leavitt path algebras as
it first may seem. Still, considering Cohn-Leavitt algebras is an elegant way to
unite considerations of both Cohn path and Leavitt path algebras. Because of
this, we also formulate the results of previous sections in terms of
Cohn-Leavitt algebras. As a consequence, each results is readily applicable to a
Cohn path or any other Cohn-Leavitt algebra without referring to the
construction of the graph $E_S$ or the isomorphism $CL_K(E,S)\cong L_K(E_S).$ 

Using relations (V), (E1), (E2) and (CK1), every nonzero element of $CL_K(E, S)$
can be represented as a finite $K$-linear combination of elements of the form 
$pq^*$ where $p$ and $q$ are paths. Thus, the involution $*$ from $K$ extends to
an involution of $CL_K(E,S)$ by $(\sum_{i=1}^n a_ip_iq_i^*)^* =\sum_{i=1}^n
a_i^*q_ip_i^*$ for paths $p_i$ and $q_i$ and $a_i\in K, i=1,\ldots, n$, making 
$CL_K(E,S)$ an involutive $K$-algebra.

We adapt Theorems \ref{positive} and \ref{faithful} to Cohn-Leavitt algebras
now. 

\begin{proposition}
Let $R$ be an involutive $K$-algebra and $t$ a canonical,
$K$-linear, $R$-valued trace on $CL_K(E,S)$.
The trace $t$ is positive if and only if  
\begin{enumerate}
\item[(P)] $\;\;\;t(v-\sum_{e\in I} \ra(e))\geq 0\;$ for all vertices $v$ and  finite subsets
$I$ of $\so^{-1}(v)$. 
\end{enumerate}

If $R$ is positive definite, then $t$ is faithful
if and only if (P), (F) and (SF) hold.
\begin{enumerate}
\item[(F)]  $t(v)> 0$ for all vertices $v.$
\item[(SF)]  $t(v-\sum_{e\in \so^{-1}(v)}\ra(e))> 0\;$ for all regular
vertices $v$ not in $S.$
\end{enumerate} 
\label{CL_positive_faithful}
\end{proposition}
\begin{proof}
Note that axiom (CK2) was not used in the proofs of Theorems \ref{positive}
and any of its preliminary results. Thus, the proof of
Theorem \ref{positive} directly transfers to the proof of the first
part of the claim. 

The proof of Theorem \ref{faithful} also directly carries over to
the proof of the second part of the claim except
for the following step of the proof which requires the use of (SF):
assuming that $t$ is such that (P), (F), and (SF)
hold, 
$$t(v-\sum_{e\in I} ee^*)=0\mbox{ implies }v=\sum_{e\in I} ee^*$$
for any vertex $v$ and any finite set $I\subseteq
\so^{-1}(v).$ To prove this step, note first that the claim trivially holds if
$v$ is a sink or $I$ is empty since the assumption $t(v)=0$ is false by (F).
Thus we can assume that $v$ is not a sink and $I$ is nonempty. In this case,
assume that $t(v-\sum_{e\in I} ee^*)=0.$ Condition (SF) implies that 
$v\notin R(E)\setminus S$ or $I\subsetneq \so^{-1}(v).$ With these
restrictions, we either have $I=\so^{-1}(v)$ and
$v\in S$ or $I\subsetneq \so^{-1}(v).$ If $v\in S$ and $I=\so^{-1}(v),$ 
$v=\sum_{e\in I} ee^*$ by (SCK2). If $I\subsetneq
\so^{-1}(v),$ there is an edge $f\in \so^{-1}(v)\setminus I.$ The
element $v-ff^*-\sum_{e\in I} ee^*$ is selfadjoint and is easily seen to be
idempotent using just (V), (E1), (E2) and (CK1). So, it is positive. Thus, we
have that 
$$t(v)\geq t(ff^*)+\sum_{e\in I}t(ee^*)>\sum_{e\in I}t(ee^*)$$ 
by
positivity of $t$ and condition (F). 
This contradicts $0=t(v-\sum_{e\in I} ee^*)=t(v)-\sum_{e\in I} t(ee^*)$ so
the case $I\subsetneq
\so^{-1}(v)$ cannot happen. Thus  $v\in S$ and $I=\so^{-1}(v)$ in which case   
 $v=\sum_{e\in I} ee^*.$ The rest of the proof of Theorem
\ref{faithful} directly transfers to the proof of the present claim. 
\end{proof}

The proofs of Proposition \ref{canonical} and  Theorem \ref{correpondence} can
also be transfered directly to Cohn-Leavitt setting after 
adjusting the definition of a graph trace as follows. 

\begin{definition}
If $R$ is a ring, an $R$-valued {\em graph
trace on $E$ relative to $S$} is a map $\delta: E^0\to R$ such that 
\begin{enumerate}
\item[(SCK2)$_\delta$] $\;\delta(v) =\sum_{e\in \so^{-1}(v)} \delta(\ra(e))\;$
for all vertices $v$ in $S.$ 
\end{enumerate}
A graph trace on $E$ relative to $R(E)$ is simply called a {\em graph
trace} on $E.$

If $R$ is an involutive $K$-algebra and $\delta$ a graph trace on $E$ relative
to $S$, then $\delta$ is {\em positive} if condition (P)$_\delta$ holds. If
$\delta$ is positive, then $\delta$ is {\em faithful} if (F)$_\delta$ and
(SF)$_\delta$ hold for  
\begin{enumerate}
\item[(SF)$_\delta$]  $\delta(v)>\sum_{e\in \so^{-1}(v)}\delta(\ra(e))\;$ for
all regular vertices $v$ not in $S.$
\end{enumerate}
\end{definition}

\begin{proposition}
Proposition \ref{canonical} and  Theorem \ref{correpondence} hold for
$CL_K(E,S)$ after every appearance of ``graph trace'' is replaced by ``graph
trace relative to $S$''. 
\label{CL_correspondence}
\end{proposition}
\begin{proof}
By considering graph traces relative to $S$ instead of graph traces, using 
(SCK2) instead of
(CK2) and (SCK2)$_\delta$ instead of (CK2)$_\delta$, we obtain the proofs of
\cite[Theorem
28]{Zak_paper} and Proposition \ref{canonical} adjusted to 
Cohn-Leavitt algebras. As a consequence of this and
Proposition \ref{CL_positive_faithful}, Theorem \ref{correpondence},
adjusted appropriately, holds for
$CL_K(E,S).$ 
\end{proof}

Our next goal is to adapt the construction from \cite[Theorem
3.7]{Muhly_Tomforde} to show that any Cohn-Leavitt algebra $CL_K(E,S)$ is
$*$-isomorphic to the Leavitt path algebra $L_K(E_S)$ of a suitable graph
$E_S$ defined via $E$ and $S$. Recall that a homomorphism $f$ of $*$-rings is
said to be a $*$-homomorphism if $f(x^*)=f(x)^*$ for every $x$ in the domain and
that a $*$-isomorphism is an isomorphism which is also a $*$-homomorphism. Also
recall that the universal property of Leavitt path algebras states that if $R$
is a $K$-algebra which contains a set $\{a_v, b_e, c_{e^*} | v\in E^0, e\in
E^1\}$ such that $a_v, b_e, c_{e^*}$ satisfy axioms (V), (E1), (E2), (CK1), and
(CK2) (such set is called  a {\em Leavitt $E$-family}) then there is a unique
$K$-algebra homomorphism $f: L_K(E)\to R$ such that $f(v) = a_v, f(e) = b_e,$
and $f(e^*) = c_{e^*}$ for all $v\in E^0$  and $e\in E^1$ (see \cite[Remark
2.11]{Abrams_Tomforde} for example). The next lemma explores the requirements
for such homomorphism $f$ to be a $*$-homomorphism.

\begin{lemma} 
For every involutive $K$-algebra $R$ with a Leavitt $E$-family
$\{a_v, b_e, c_{e^*} | v\in E^0, e\in E^1\}$ such that $a_v^*=a_{v^*}$ and 
$b_e^*=c_{e^*},$ there is a unique
$K$-algebra $*$-homomorphism $f: L_K(E)\to R$ such
that $f(v) = a_v, f(e) = b_e,$ (thus $f(e^*) = c_{e^*}$) for all
$v\in E^0$  and $e\in E^1$. 
\label{universal_star} 
\end{lemma}
\begin{proof}
Since $\{a_v, b_e, c_{e^*} | v\in E^0, e\in E^1\}$ is a Leavitt
$E$-family, there is a unique
$K$-algebra homomorphism $f: L_K(E)\to R$ such that $f(v) = a_v, f(e) = b_e,$
and $f(e^*) =
c_{e^*}.$ We claim that under assumption that $a^*_v=a_v$ and
$b_e^*=c_{e^*},$ the map $f$ is a $*$-homomorphism. 

Since $f$ is additive and  $K$-linear, it is sufficient to prove that
$f(x^*)=f(x)^*$ if $x=pq^*$ where $p$ and $q$ are
paths with $\ra(p)=\ra(q).$ The condition $a^*_v=a_v$ proves this statement for 
$|p|=|q|=0$ and the condition $b_e^*=c_{e^*}$ implies that
$f(e)^*=f(e^*)$ for every edge $e.$  Assuming that the statement holds
for any path $p$ with $|p|<n$ and $q$ with $|q|=0,$ let us prove it if $p$ has
length $n$ and $q$ length 0. In this case $p=er$ for some edge $e$ and
path $r$ with $\ra(e)=\so(r)$ and $|r|<n$ so that $f(r)^*=f(r^*)$ by the
induction hypothesis and $pq^*=er.$ Thus
$f(pq^*)^*=f(er)^*=(f(e)f(r))^*=f(r)^*f(e)^*=f(r^*)f(e^*)=f(r^*e^*)=f((er)^*
)=f((pq^*)^*).$ 
 
Now, assuming the statement for $pq^*$ with $|q|<m$,
let us prove it for $pq^*$ with $|q|=m.$ In this case
$q=er$
for some edge $e$ and path $r$ with $\ra(e)=\so(r)$ and $|r|<m$ and so 
$f(pq^*)^*=f(pr^*e^*)^*=(f(pr^*)f(e^*))^*=f(e^*)^*f(pr^*)^*=f(e)f(rp^*)=
f(erp^*)=f((pr^*e^*)^*)=f((pq^*)^*).$ 
\end{proof}

We shall use Lemma \ref{universal_star} to show that a Cohn-Leavitt algebra
$CL_K(E,S)$ is $*$-isomorphic to the Leavitt path algebra $L_K(E_S)$ where
$E_S$ is the graph obtained from $E$ and $S$ as in 
\cite[Theorem 3.7]{Muhly_Tomforde}. First, we
recall the construction of $E_S$ from
\cite[Definition 3.6]{Muhly_Tomforde} and the map
$\phi$ defined on the vertices, edges and ghost edges of $E_S$ with values in
$CL_K(E,S)$ which creates a Leavitt $E_S$-family in $CL_K(E,S)$. 

Let $E^0_S=E^0\cup\{v' | v\in R(E)\setminus S\}$ and $E^1_S=E^1\cup \{e' | e\in
E^1$ with $\ra(e)\in R(E)\setminus S\}.$ The maps $\so$ and $\ra$
in $E_S$ are the same as in $E$ on $E^1$ and such that $\so(e')=\so(e)$ and
$\ra(e')=\ra(e)'$ for any added edge $e'.$ 

Define $\phi$ on the vertices of $E_S$ by
$\phi(v)=v$ if $v\notin R(E)\setminus S$,
$\phi(v)=\sum_{e\in \so^{-1}(v)} ee^*$ if $v\in R(E)\setminus S,$ and
$\phi(v')=v-\sum_{e\in \so^{-1}(v)} ee^*$ for $v\in R(E)\setminus S.$
Note that this automatically gives us $\phi(w)^*=\phi(w)$ for every vertex $w$
of $E_S.$
Define $\phi$ on the edges of $E_S$ by
$\phi(e)=e\phi(\ra(e))$ for $e\in E^1$ and
$\phi(e')=e\phi(\ra(e)')$ for $e\in
E^1$ such that $\ra(e)\in R(E)\setminus S.$ 
Moreover, 
define $\phi$ on the ghost edges of $E_S$ by
$\phi(f^*)=\phi(f)^*$ for every edge $f$ of $E_S$.

\begin{lemma}
\label{star_iso}
The map $\phi$ extends to a $*$-isomorphism $\phi: L_K(E_S)\cong CL_K(E,S).$ 
\end{lemma}
\begin{proof}
It can be directly checked that the map $\phi$ defined as above is
such that the images $\phi(w)$, $\phi(f),$ and
$\phi(f^*)$ for $w\in E^0_S$ and $f\in E^1_S$ satisfy (V), (E1), (E2), (CK1),
and (CK2). Since  $\phi(w)^*=\phi(w)$ for $w\in E^0_S$ and
$\phi(f^*)=\phi(f)^*$ for $f\in E^1_S$, the set $\{\phi(w), \phi(f),
\phi(f^*) \ |\ w\in E^0_S, 
f\in E^1_S\}$ satisfies the assumptions of Lemma \ref{universal_star} and
so $\phi$ uniquely extends to a $K$-algebra
$*$-homomorphism of $L_K(E_S)$ to $CL_K(E,S)$.   

Note that $\phi$ is onto since the vertices, edges and ghost edges of $E$ are in
the image of $\phi.$ Indeed, $v=\phi(v+v')$ for $v\in R(E)\setminus S$ 
and $v=\phi(v)$ for a vertex $v\notin R(E)\setminus S.$ 
Also, if $e\in E^1,$ $e=\phi(e+e')$ for $\ra(e)\in R(E)\setminus
S$ and $e=\phi(e)$ otherwise. From this it follows that the ghost
edges of $E^1$ are in the image of $\phi$ as well.  

Since $\phi$ preserves the grading on the vertices, edges and ghost edges,
$\phi$ is
a graded homomorphism by construction. Thus $\phi$ is a monomorphism by the
Graded Uniqueness Theorem (\cite[Theorem 4.8]{Tomforde}). Note that $E$ is
assumed to be countable in \cite{Tomforde} but the proof of \cite[Theorem
4.8]{Tomforde} does not use this fact.
\end{proof}

Using Lemma \ref{star_iso} and \cite[Proposition
2.4]{GonzaloRangaLia}, we note that the following
conditions are equivalent. \cite[Proposition
2.4]{GonzaloRangaLia} states that the three conditions, analogous to the three
conditions below but formulated for Leavitt path algebras, are equivalent. 
\begin{enumerate} 
\item The involution on $K$ is positive definite.
\item The involution on $CL_{K}(E, S)$ is positive definite for every graph $E$
and $S\subseteq R(E)$.
\item The involution on $CL_{K}(E, S)$ is positive definite for some graph $E$
and $S\subseteq R(E)$.
\end{enumerate}

Using the implication (1) $\Rightarrow$ (2), we show the Cohn-Leavitt
version of Proposition \ref{positively_faithful_on_LPA}. 
\begin{corollary}
If $R$ is a positive definite, involutive
$K$-algebra and $t:CL_K(E,S)\to R$ is a faithful, $K$-linear
map, then $K$ and $CL_K(E,S)$ are positive definite.  
\end{corollary}
\begin{proof}
If $\phi$ is the isomorphism from Lemma \ref{star_iso}, then the
composition $t\circ\phi$ satisfies the assumption of Proposition
\ref{positively_faithful_on_LPA} so $K$ is positive definite by Proposition
\ref{positively_faithful_on_LPA}. Then  
$CL_K(E,S)$ is positive definite by Proposition
\ref{positively_faithful_on_LPA} and the implication (1) $\Rightarrow$ (2) 
above.
\end{proof}
 
Continuing on towards proving the main result of this section,
we note
that any Cohn-Leavitt algebra is a ring with local units (to see
that use the same arguments as before when considering Leavitt path
algebras). Thus, the definitions of directly finite and finite locally
unital rings
apply to Cohn-Leavitt path algebras as well. In addition to forcing the
underling graph $E$ to be no-exit, the direct finiteness of
a Cohn-Leavitt path algebra $CL_K(E,S)$ also forces the vertices of all cycles
of $E$ to be in $S.$ Using Lemma \ref{star_iso} and \cite[Theorem
3.3]{GonzaloLia}, the next result shows that these conditions
are also sufficient for direct finiteness if $E$ is finite.  

\begin{proposition}
If $CL_K(E,S)$ is (directly) finite, then the following two
conditions hold. 
\begin{enumerate}
\item $E$ is no-exit. 

\item If a vertex is in a cycle, then it is in $S$.
\end{enumerate}

If $E$ is a finite graph, conditions (1) and (2) imply that $CL_K(E,S)$
is directly finite. 
\label{CL_finite_implies_no-exit}
\end{proposition}
\begin{proof}
The proof of Proposition \ref{finite_implies_no-exit} demonstrates part
(1) since it does not use axiom (CK2). 

To show part (2), assume that a vertex $v$ is in a cycle. Then $v$
is not a sink
nor it is an infinite emitter since $E$ is no-exit. Thus $v$ is regular. 
If $v$ is not in $S$, then $v$ is in a cycle with an exit in the
graph $E_S$ by construction of $E_S.$ Thus, $L_K(E_S)$ is not
finite by Proposition \ref{finite_implies_no-exit}. Since $CL_K(E,S)$ is
$*$-isomorphic to $L_K(E_S)$ by Lemma \ref{star_iso}, $CL_K(E,S)$ is not  
finite as well. This contradicts the assumption so $v$ is in $S.$

Assume now that $E$ is a finite graph satisfying (1) and (2). Conditions (1) and
(2) imply that the graph $E_S$ is no-exit by construction of $E_S$. Thus,
$L_K(E_S)$ is directly finite by \cite[Theorem 3.3]{GonzaloLia} and so 
$CL_K(E,S)\cong L_K(E_S)$ is directly finite as well. 
\end{proof}

The last ingredient needed for our proof of Theorem \ref{no-exit} is the
construction from \cite[Definition 3.4, Propositions 3.5 and 3.6]{Ara_Goodearl}.
In \cite{Ara_Goodearl}, the authors consider Cohn-Leavitt algebras of separated
graphs. Since we consider non-separated (i.e. trivially separated)
graphs, we present \cite[Definition 3.4, Propositions 3.5 and 3.6]{Ara_Goodearl}
below assuming the trivial partition $\{\so^{-1}(v)\}$ for every $v$. With
this restriction, \cite[Definition 3.4]{Ara_Goodearl} can
be stated as follows. 

Let $E$ be a graph with $S\subseteq R(E)$ and $F$ a subgraph of $E$ with
$T\subseteq R(F).$ We say that $(F, T)$ is a {\em complete subobject} of
$(E,S)$ if $T\subseteq S$ and the following holds.
\begin{enumerate}
\item[(C)] If $v\in
S\cap F^0$ with $\so^{-1}_E(v)\cap F^1\neq \emptyset$ then
$\so^{-1}_F(v)=\so^{-1}_E(v)$ and $v\in T.$
\end{enumerate}

If $T=R(F)$ and $S=R(E),$ this agrees with the definitions of a
complete subgraph for row-finite graphs from \cite[Section
3]{Ara_Moreno_Pardo} and for countable graphs from 
\cite[Definition
9.7]{Abrams_Tomforde}.
Note that conditions $T\subseteq S$ and (C) imply that 
$T=S\cap\{v\in F^0 | \so^{-1}_E(v)\cap F^1\neq \emptyset\}.$
Indeed, every vertex $v\in T$ is necessarily in $S$ and it emits (finitely many) edges in $F$ so that  $\so^{-1}_E(v)\cap F^1\neq
\emptyset$. The converse holds by (C).

\begin{proposition}
If $G$ is a finite subgraph of a graph $E$ and $S\subseteq R(E),$ there is a
complete subobject
$(F,T)$ of $(E,S)$ such that $F$ is finite, $G$ is a subgraph of $F$ and
$CL_K(F,T)$ is a $K$-subalgebra of $CL_K(E,S).$
\label{complete_subgraphs} 
\end{proposition}
\begin{proof}
Following \cite[Proposition 3.5]{Ara_Goodearl}, we define the graph $F$ and the
set $T\subseteq R(F)$ as follows.
\[
\begin{array}{ll}
F^0=G^0\cup \{\;\ra_E(e)\;\; |\, e\in E^1, &\so_E(e)\in G^0\cap S \mbox{ and
}\so^{-1}_E(\so_E(e))\cap G^1\neq\emptyset\},\\
F^1 = G^1\cup\{\; e\in E^1 | &\so_E(e)\in G^0\cap S \mbox{ and
}\so^{-1}_E(\so_E(e))\cap G^1\neq \emptyset\},\mbox{ and}
\end{array}
\]
\[T=S\cap\{v\in F^0 | \so^{-1}_E(v)\cap F^1\neq \emptyset\}.\]

With these definitions, $F$ is finite, $T\subseteq S,$ and if $v\in
S\cap F^0$ with $\so^{-1}_E(v)\cap F^1\neq \emptyset,$ then $v\in T$ and
$\so^{-1}_E(v)\cap G^1\neq \emptyset$ so $\so^{-1}_F(v)=\so^{-1}_E(v)$ by
definition of $F^1.$ Thus condition (C) holds.
It is straightforward to see that axioms (V), (E1), (E2) and (CK1) are
compatible in $CL_K(E,S)$ and $CL_K(F,T).$ If $v\in T,$ then
$\so^{-1}_F(v)=\so^{-1}_E(v)$ is finite. Thus, $\sum_{e\in
\so^{-1}_F(v)}ee^*=\sum_{e\in
\so^{-1}_E(v)}ee^*=v$ so (SCK2) is compatible as well.
The inclusion of $(F,T)$ into $(E,S)$ induces the inclusion of the basis
of $CL_K(F,T)$ into the basis of $CL_K(E,S)$ described in \cite[Propositions
2.7 and 3.6]{Ara_Goodearl}. This induces an embedding of $CL_K(F,T)$ into
$CL_K(E,S).$
\end{proof}

In case when $S=R(E),$ $E$ has an infinite emitter $v$ and $G$ is a subgraph
consisting of $v$ with finitely many edges $v$ emits together their ranges,
the complete subobject $(F,T)$ from Proposition \ref{complete_subgraphs} is such
that $v\in R(F)$ but $v\notin T.$ Thus
$CL_K(F,T)$ is a $K$-subalgebra of $L_K(E)$ while $L_K(F)$ is not. Cases like
this one make the consideration of Cohn-Leavitt algebras necessary in the proof 
of our next result, the main result of this section. 

\begin{theorem}
The following conditions are equivalent. 
\begin{enumerate}
\item $L_K(E)$ is directly finite. 

\item $L_K(E)$ is finite.

\item $E$ is no-exit.   
\end{enumerate}
\label{no-exit}
\end{theorem}
\begin{proof}
(1) trivially implies (2). (2) implies (3) by Proposition
\ref{finite_implies_no-exit}. 

To show that (3) implies (1), assume that $x,y\in L_K(E)$ are such
that $xu=ux=x,$ $yu=uy=y$ and $xy=u$ for some idempotent $u$ of $L_K(E).$
If $x,$ $y,$ and $u$  are finite $K$-linear combinations of elements of the form
$p_iq_i^*,$ $i=1,\ldots, n$ for
some paths $p_i$ and $q_i,$ let $G$ be the subgraph of $E$ such that $G^0$ is
the set of all vertices
appearing in paths $p_i$ and $q_i,$ $i=1,\ldots, n,$ and $G^1$ is the
set, possibly empty, of all edges of paths $p_i$ and $q_i$. Let $(F, T)$ be
the complete subobject of $(E, R(E))$ generated by $G$ from Proposition
\ref{complete_subgraphs}. The graph $F$ is no-exit
since $E$ is no-exit. If a vertex $v$ is in a cycle of $F$ then it
emits a single edge both in $F$ and in $E.$  
Thus $v$ is regular and $\so^{-1}_E(v)\cap F^1\neq
\emptyset$ so $v$ is in $T$ by definition of $F$ and $T$. This enables us to use
Proposition \ref{CL_finite_implies_no-exit} and to conclude that $CL_K(F,T)$ is
directly finite. 
By construction $x,$ $y$ and $u$ are elements of $CL_K(F,T)$ and so $u^2=u,$  
$xu=ux=x,$ $yu=uy=y$ and $xy=u$ are relations in $CL_K(F,T)$ as well. Since
$CL_K(F,T)$ is directly finite, these relations imply $yx=u$. The relation
$yx=u$ then holds in $L_K(E)$ as well. So $L_K(E)$ is directly finite too. 
\end{proof}

Theorem \ref{no-exit} has the following corollary.  

\begin{corollary} 
The following conditions are equivalent. 
\begin{enumerate}
\item $CL_K(E,S)$ is directly finite. 

\item $CL_K(E,S)$ is finite.

\item $E$ is no-exit and vertices of every cycle are in $S$.  
\end{enumerate}

In particular, a Cohn path algebra $C_K(E)$ is (directly) finite if and only
if $E$ is acyclic.
\label{CL_no-exit}
\end{corollary}
\begin{proof}
(1) trivially implies (2) and (2) implies (3) by Proposition
\ref{CL_finite_implies_no-exit}. To show that (3) implies (1), note that (3)
implies that the graph $E_S$ is no-exit. Thus $L_K(E_S)$ is
directly finite by Theorem \ref{no-exit} and then so is $CL_K(E,S)\cong
L_K(E_S).$ 

The equivalence of conditions (1), (2) and (3) with
$S=\emptyset$ shows the last sentence.  
\end{proof}

In the final part of the paper, we focus on relations between the
following
three conditions. 
\begin{enumerate}
\item $L_K(E)$ is locally noetherian. 

\item $L_K(E)$ admits a faithful trace. 

\item $L_K(E)$ is directly finite. 
\end{enumerate}

Recall that a ring $T$ is locally left (right) noetherian if for every finite
set $F$ of $T$, there is an idempotent $e\in T$ such that $eTe$ contains $F$ and
$eTe$ is left (right) noetherian. By \cite[Theorem 3.7]{AAPM}, a Leavitt path
algebra is locally left noetherian if and only if it is locally right
noetherian and we simply say it is locally noetherian in this case. Recall
also that an infinite path of a graph is a sequence of edges $e_1e_2\ldots$ such
that $\ra(e_i)=\so(e_{i+1})$ for all $i=1,2,\ldots$. An 
infinite path $p$ is an \emph{infinite sink} if it has
no cycles or
exits. An infinite path $p$ \emph{ends in a sink} if there is $n\geq 1$
such that the subpath
$e_ne_{n+1}\hdots$ is an infinite sink, and $p$
\emph{ends in a cycle} if there is $n\geq 1$ and a
cycle $c$ of positive length such that the subpath
$e_ne_{n+1}\hdots$ is equal to the path $cc\hdots$.  
\cite[Theorem 3.7]{AAPM} asserts that the following conditions are equivalent for
every graph $E$.
\begin{enumerate}
\item $L_K(E)$ is locally noetherian.

\item $E$ is a no-exit graph such that every infinite path ends either in a
sink or in a cycle. 
\end{enumerate}

Proposition \ref{faithful_implies_finite}, Theorem \ref{no-exit},
and \cite[Theorem 3.7]{AAPM} infer the implications and equivalences in the
diagram below. 

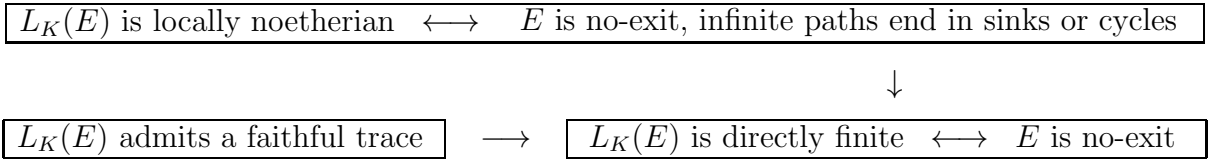
\[
\begin{array}{|c|}\hline
L_K(E)\mbox{ is locally noetherian} 
\;\;\longleftrightarrow \;\;
\begin{array}{c}
E\mbox{ is no-exit, infinite paths end in sinks or cycles}
\end{array}
\\ \hline
\end{array}
\]
\[\begin{array}{ccc}
\hskip6cm &\hskip1cm &\downarrow \\
\end{array}\]
\[\begin{array}{ccc}
\begin{array}{|c|}\hline
L_K(E)\mbox{ admits a faithful trace}\; \\\hline
\end{array} 
\;&\longrightarrow&  \;
\begin{array}{|c|}\hline
\;L_K(E)\mbox{ is directly
finite} \;\;\longleftrightarrow\;\; E\mbox{ is no-exit}\;\;\\\hline
\end{array} \\
\end{array}\]
\smallskip

The next three examples show that both implications are strict and the two
conditions ``$L_K(E)$  is locally noetherian'' and ``$L_K(E)$ admits a faithful
trace'' are independent. 

\begin{example}
\label{direct_finite_not_noetherian_no_trace}
If $E$ is any no-exit graph which has an
infinite path not ending in a sink or a cycle, then $L_K(E)$
is directly
finite and not locally noetherian. For example, the graph $E$ below has this
property.
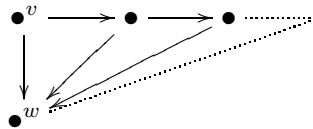
\[\xymatrix{{\bullet}^{v}
\ar [r]\ar[d] & {\bullet} \ar [r] \ar[dl] &
{\bullet} \ar@{.}[r] \ar[dll] & \ar@{.}[dlll]\\ {\bullet}^{w}
}\]
There is no faithful $\Cset$-valued graph trace on $E$ since there are
infinitely many paths from $v$ to $w$ forcing $\delta(w)$ to be zero for any
positive graph trace $\delta$ on $E$. This demonstrates
that a directly finite Leavitt path algebra may not admit a faithful trace.
\end{example}

\begin{example} 
\label{noetherian_no_trace}
Let $E$ be the graph with two
vertices $v$ and $w$ and infinitely many edges from $v$ to $w$ as
represented below. The Leavitt path algebra of
this graph over any field is locally noetherian since $E$ satisfies the
required graph-theoretic condition. However, there
is no faithful trace on $L_{\Cset}(E)$ (nor $L_{\Rset}(E)$ as noted also
in \cite[Example 35]{Zak_paper}) since there are infinitely many paths from
$v$ to $w.$ 
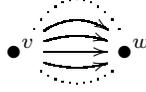
$$\xymatrix{{\bullet}^{v} \ar@{.} @/_1pc/ [r] _{\mbox{ } } \ar@/_/ [r] \ar [r]
\ar@/^/ [r] \ar@/^1pc/ [r] \ar@{.} @/^20pt/
[r] & {\bullet}^{w}}$$
\end{example}

\begin{example} 
\label{has_trace_not_noetherian}
The graph $E$ represented below is such that there is an infinite
path which does not end in a sink or a cycle so $L_K(E)$ is not locally
noetherian for any field $K$.
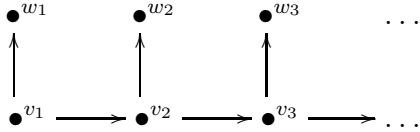
\[\xymatrix{
\bullet^{w_1} & \bullet^{w_2} & \bullet^{w_3}  & \dots & \\
\ar@<1ex>[u] \bullet^{v_1} \ar[r] & \ar@<1ex>[u]\bullet^{v_2}  \ar[r]& \ar[r]
\ar@<1ex>[u]\bullet^{v_3}& \ar@{.}\dots
}\]
On the other hand,
there is a faithful, $\Cset$-valued graph trace given
by 
\[\delta(v_n)=\frac{1}{2^{n-1}}\;\;\;
\mbox{ and }\;\;\;\delta(w_n)=\frac{1}{2^n}\;\;\;\mbox{ for }n=1,2,\ldots\]
The graph trace $\delta$ extends to a faithful trace by Theorem
\ref{correpondence}.
\end{example}

It is interesting to note that both conditions ``$L_K(E)$  is locally
noetherian'' and ``$L_K(E)$ is directly finite'' have been characterized by
graph-theoretic conditions. We wonder if a graph-theoretic characterization can
be found for the condition ``$L_K(E)$ admits a faithful trace'' if $K$ is
positive definite. If $E$ is a row-finite graph in which every infinite path
ends in a sink or a cycle, \cite[Theorem 33]{Zak_paper} shows that $L_K(E)$
admits a faithful trace if and only if $E$ is no-exit. However the graph $E$
from Example \ref{has_trace_not_noetherian} does not fall under the class of
graphs covered by \cite[Theorem 33]{Zak_paper} and $L_{\Cset}(E)$ admits a
faithful trace. So, we wonder if a general graph-theoretic characterization is
possible. More precisely, we propose the following.  
\begin{question}
Find a graph-theoretic condition on $E$ which is equivalent to the
condition that the Leavitt path algebra $L_K(E)$ over  a positive
definite field $K$ admits a faithful
trace.  
\label{question1}
\end{question}
A similar question was also raised in \cite{Pask-Rennie} for graph
$C^*$-algebras. \cite[Lemma 3.5 and Corollary 3.7]{Pask-Rennie} show that the
following condition is necessary for a row-finite graph $E$ to admit a
$\Cset$-valued faithful graph trace. 
\begin{enumerate}
\item There are finitely many paths from any vertex to any other vertex, an
infinite sink or a loop.  
\end{enumerate}
If $E$ is row-finite, \cite[Proposition 3.8]{Pask-Rennie} lists the following
condition as sufficient.
\begin{enumerate}
\item[(2)] There is a finite upper bound for the number of paths from
any vertex to any other vertex, an infinite sink or a loop and every infinite
path ends in a loop or a sink. 
\end{enumerate}
Condition (2) is not necessary as the following example shows.

\begin{example}
Let $E$ be the graph with vertices $v$ and $w_n,$ $n=1,2,\ldots,$ and $n$
edges from $w_n$ to $v$ for every $n.$ In the diagram below, the numbers next to
the arrows indicate the number of edges
from vertices $w_n$ to $v$.
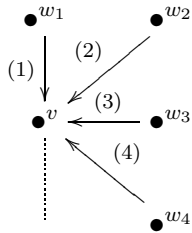
$$\xymatrix{ & {\bullet}^{w_1}\ar[d]_{(1)} & {\bullet}^{w_2}\ar[dl]_{(2)}  \\  &
{\bullet}^{v} 
   \ar@{.}[d]
\ar@{}[dl] & {\bullet}^{w_3}\ar[l]_{(3)} \\ &  & {\bullet}^{w_4}\ar[ul]_{(4)}}$$
The graph $E$ is row-finite, it does not satisfy condition (2), and there is a
faithful, $\Cset$-valued graph trace on $E$ given by 
\[\delta(v)=1\;\;\;\mbox{ and
}\;\;\;\delta(w_n)= n\;\mbox{ for }n=1,2,\ldots.\]
In fact, if $K$ is any positive definite field, the conditions above define a
faithful, $K$-valued graph trace on $E$. 
\end{example}
\medskip

\noindent{\bf Acknowledgments.}
The author is grateful to Zachary Mesyan for his valuable suggestions on an early
outline of the manuscript and, in particular, for pointing out a gap in the
proof of Theorem \ref{positive} from that version of the manuscript.

\end{document}